\documentclass[11pt]{amsart}
\usepackage{verbatim, amscd, epsfig,amsmath,amsthm,amstext,amssymb,hyperref}
\usepackage[all]{xy}
\theoremstyle{plain}
\newtheorem*{theorem*}{Theorem}
\newtheorem{theorem}{Theorem}[section]
\newtheorem{lemma}[theorem]{Lemma}
\newtheorem{cor}[theorem]{Corollary}
\newtheorem{corollary}[theorem]{Corollary}

\newtheorem{rem}[theorem]{Remark}

\theoremstyle{definition}
\newtheorem{definition}[theorem]{Definition}

\renewcommand{\Re}{{\rm Re}\,}
\renewcommand{\Im}{{\rm Im}\,}
\newcommand{\del}{\partial}
\newcommand{\delbar}{\bar{\del}}
\newcommand{\todo}[1]{\marginpar{\emph{to do:#1}}}

\newcommand{\Li}{\underline{e\mathbb{H}}}
\newcommand{\eigenline}{eigenline }

\newcommand{\R}{\mathbb{ R}}
\newcommand{\C}{\mathbb{ C}}
\newcommand{\Z}{\mathbb{ Z}}

\newcommand{\T}{\mathcal{T}}
\renewcommand{\H}{\mathbb{ H}}

\renewcommand{\P}{\mathbb{ P}}
\newcommand{\HP}{\H\P}
\newcommand{\CP}{\C\P}
\newcommand{\invers}{^{-1}}
\DeclareMathOperator{\End}{End}
\DeclareMathOperator{\Hom}{Hom}

\DeclareMathOperator{\ad}{ad}
\DeclareMathOperator{\tr}{tr}
\DeclareMathOperator{\im}{im}
\DeclareMathOperator{\Span}{span}

\DeclareMathOperator{\Harm}{Harm}
\begin{document}

\title
[Darboux transforms and spectral curves]{Darboux transforms and spectral curves of constant mean curvature surfaces revisited}
\author{E. Carberry, K. Leschke, F. Pedit}
\date{\today}

\address{Emma Carberry\\School of Mathematics and Statistics\\University of Sydney \\NSW, 2006,
Australia. }

\address{
Katrin Leschke\\ Dept. of Mathematics\\ University of Leicester\\ University Road\\ Leicester LE1 7RH, UK.}

\address{Franz Pedit\\Mathematisches Institut\\
Auf der Morgenstelle 10\\
72076 Tübingen\\
Germany\\
and
Department of Mathematics and Statistics\\
University of Massachusetts Amherst\\
Amherst, MA    01003-9305, USA
}

\email{emma.carberry@sydney.edu.au,
 katrin.leschke@mcs.le.ac.uk, \linebreak
pedit@mathematik.uni-tuebingen.de
}

\maketitle

\section{Introduction}
Surfaces of non--zero constant mean curvature in Euclidean $3$--space
are studied from a variety of different view points. These surfaces
are critical with respect to the variation of area with constrained
volume and their Euler--Lagrange equation is an important example of a
geometric non--linear elliptic partial differential equation
\cite{Wente:86, Kapouleas:91, mp, KGB}.
Constant mean curvature surfaces also have a deep connection with the
theory of integrable systems since the Gauss--Codazzi equation is a
well known soliton equation, namely the sinh--Gordon equation.  This
has led to a complete classification of constant mean curvature tori
in terms of periodic linear flows on Jacobians of hyperelliptic
algebraic curves \cite{PS:89, Hitchin:90, Bobenko:91, EKT:93,
  Jaggy:94}.

The essential ingredient of the integrable systems approach is that a
surface $f\colon M\to\R^3$ of constant mean curvature has an {\em
  associated $ S ^ 1$-family} of constant mean curvature surfaces
$f^{\mu}$, obtained by rotating the Hopf differential of $f$ by $\mu
\in S^1$.  The surfaces $f^{\mu}$ generally develop translational and rotational 
 periods in $\R^3$ if $M$ has
non--trivial topology. Extending the circle parameter, also called the
{\em spectral parameter}, to $\mu \in \C\P^1$ one obtains a rational
family of flat ${\bf SL}(2,\C)$ connections $\nabla^{\mu}$ over the
surface $M$ with simple poles at $\mu=0,\infty$. This family is
unitary along the unit circle $\mu\in S^1$ where it describes the
associated family of constant mean curvature surfaces $f^{\mu}$. When
$M=T^2$ is a torus, the holonomy representation $H^{\mu}$ of the
family $\nabla^{\mu}$ with respect to a chosen base point $ p $ on
$T^2$ is abelian and hence has simultaneous eigenlines.  These
eigenlines define a hyperelliptic algebraic curve $\Sigma _ e $ over
the $\mu $--parameter space $\C\P^1$, together with a holomorphic line
bundle $\mathcal E (p) $ over $\Sigma_e$.  Changing the base point
$p\in T^2$ does not change this {\em eigenline spectral curve} $\Sigma
_ e $, but the {\em eigenline bundles} $\mathcal E (p) $ sweep out a
$2$--dimensional subtorus of the Jacobian of $\Sigma_e$.

Recently a more general notion of a spectral curve has been introduced
\cite {conformal_tori, Taimanov, Schmidt} for any conformally immersed
torus $f\colon T^2\to S^4$. This {\em multiplier spectral curve}
$\Sigma _ m $ does not rely upon the existence of a family of flat
connections and 
it arises rather geometrically as a desingularisation of the
set of all Darboux transforms of $f$.
By a Darboux transform of $f$ we mean a conformally immersed torus
$\hat f\colon T^2\to S^4$ for which there is a $2$--sphere congruence
along $f$ touching $f$ and half--touching $\hat f$.  Darboux
transforms include the {\em classical} Darboux transforms for which
$\hat f$ also touches, rather than merely half--touches, the said
sphere congruence. In the classical case both surfaces $f$ and $\hat f $ are
isothermic.  Analytically the multiplier spectral curve of a
conformally immersed torus $f$ is given by the possible holonomies, or
``multipliers'', of quaternionic holomorphic sections for the
quaternionic holomorphic structure $D$ induced by $f$ on the
quaternionic bundle $V/L$. Here $L$ is the pull--back under $f$ of the
tautological bundle over $S^4=\H\P^1$ and $V$ is the trivial
$\H^2$--bundle.  Generically there is (up to scale) exactly one
quaternionic holomorphic section for a given multiplier in $\Sigma_m$
and one thereby obtains a $ T ^ 2 $-family of holomorphic line
bundles, the {\em kernel bundles}, over the multiplier spectral
curve. Again, in the case when $\Sigma_m$ has finite genus, this $T^2$
family of holomorphic line bundles sweeps out a subtorus \cite{ana} of the
Jacobian of $\Sigma_m$.  The Darboux transform of $f$ corresponding to
a multiplier $h\in\Sigma_m$ is then given by $\hat f=\hat{\varphi}\H$,
where $\hat\varphi$ is the prolongation to $V$ of the quaternionic
holomorphic section $\varphi$ of $V/L$ with holonomy $h$.

This paper discusses this general approach to spectral curves in the
context of constant mean curvature tori in $\R^3$: what are the
geometric properties of the Darboux transforms of a constant mean
curvature torus? Are the two spectral curves the same and how do the
eigenline and kernel bundles relate?

The first question we can answer even locally: given a (simply
connected) constant mean curvature surface $f\colon M\to\R^3$ any
parallel section of the flat connection $\nabla^{\mu}$ is quaternionic
holomorphic since the connections $\nabla^{\mu}$ all induce the same
quaternionic holomorphic structure $D$ on $V/L$. The prolongation of
any such parallel section is therefore a Darboux transform of $f$,
which we call a {\em $\mu$-Darboux transform}.  Note that there is a
$\C\P^1$-worth of parallel sections to a given $\mu$. We show that all
$\mu$--Darboux transforms of the constant mean curvature surface $f\colon M\to\R^3$
are again constant mean curvature surfaces, albeit in a parallel translated $\R^3\subset \R^4$.  
Amongst the classical Darboux
transforms of $f$ in $\R^3$, those which have constant mean curvature form a $3$-dimensional hyper-surface 
\cite{darboux_isothermic}. We show that this hyper-surface
consists entirely of $\mu$-Darboux transforms with $\mu\in\R\setminus\{0,1\}$.
For non-real $\mu$ the only other $\mu$-Darboux transforms  which are classical occur for unitary $\mu$ and 
give the parallel constant mean curvature surface. 

The second question concerns the global existence of Darboux
transforms of a constant mean curvature torus $f\colon T^2\to \R^3$.
Away from the points $P_0, P_{\infty}$ over $\mu=0,\infty$ a point $ x
$ on the eigenline spectral curve $\Sigma_e$ of $f$ is described by a
parallel section of $\nabla^{\mu}$ with holonomy.  This section
therefore gives a $\mu$-Darboux transform $\hat f^{x}\colon T^2\to
S^4$ defined on the same $2$--torus and its holonomy defines a map $h
$ into the multiplier spectral curve $\Sigma_m$. This map does not
extend to an isomorphism of the two curves; the eigenline curve is
algebraic and generically smooth whereas the multiplier curve is
always singular and may have infinite arithmetic genus.  This
discrepancy is resolved by desingularising: the lift of the map $h $
to the normalisations extends holomorphically, thereby compactifying
the normalisation of the multiplier spectral curve and yielding a
biholomorphism of the two desingularised curves.  The $\mu$-Darboux
transforms limit to the original constant mean curvature torus at
$P_0$ and $P_{\infty}$, yielding a geometric method for recovering the
original surface. A similar result has been proven in the more general
context of conformally immersed tori in the $4$-sphere having finite
spectral genus \cite{conformal_tori}. However, we utilise the
existence of a family of flat connections in the constant mean
curvature case to give a considerably simpler proof.  Since the kernel
and eigenline bundles lift to holomorphic line bundles which agree by
construction away from a discrete set of points, they are the same
bundle on the identified normalisations of the two spectral curves.
This implies that any Darboux transform of a constant mean curvature
torus $f\colon T^2\to \R^3$ which is para\-meter\-ised by the
multiplier curve is a $\mu$-Darboux transform and hence also a torus
of constant mean curvature in $\R^3$. However not all Darboux
transforms or even all classical Darboux transforms of a constant mean
curvature torus must again have constant mean curvature \cite
{Bernstein:01, rectangular}.  The space of all Darboux transforms may,
in addition to the multiplier curve, contain countably many
quaternionic and complex projective spaces.  There may similarly be
many $\mu$--Darboux transforms which are not parameterised by the
eigenline curve, and in the appendix we study the simple example of
the standard cylinder and show that these correspond to the adding of
bubbletons to the original constant mean curvature surface.


\section {Darboux Transformations}

\renewcommand{\hat}{\widehat}
We model the conformal geometry of the 4--sphere by the
quaternionic projective line $S^4=\HP^1$ on which the group of
orientation preserving M\"obius transformations acts by ${\bf GL}(2,\H)$. A
map $f: M\to S^4$ can be considered as a line subbundle $L\subset V$
of the trivial $\H^2$ bundle $V= \underline{\H}^2$, where the fibers
of $L$ are given by $L_p=f(p)$ for $p\in M$.  In other words,
$L=f^*\T$ is the pullback of the tautological line bundle $\T$ over
$\HP^1$. Identifying the tangent bundle of $\HP^1$ with
$\Hom(\T,\underline\H^2/\T)$ the derivative of $f$ is given by
\begin{equation}
\label{eq:delta}
\delta = \pi d|_L\in\Omega^1(\Hom(L,V/L))\,,
\end{equation}
where $d$ is the trivial connection on $V$ and $\pi: V\to V/L$ is the
canonical projection. Throughout the paper we denote by $\Hom(W_1,
W_2)$ the real vector space of quaternionic linear maps between
quaternionic (right) vector spaces $W_1$ and $W_2$. An immersion
$f\colon M \to S^4$ is conformal \cite{coimbra} if and only if there
exist complex structures $J\in\Gamma(\End(V/L))$ on $V/L$ and
$\widetilde J\in\Gamma(\End(L))$ on $L$ such that
\begin{equation}
\label{eq:conformality}
*\delta = J \delta = \delta\widetilde J\,,
\end{equation}
where $*$ is the conformal structure on $T^*M$.

An oriented round 2--sphere in $S^4=\HP^1$ is described by a complex
structure $S\in\End(\H^2)$, $S^2=-1$: points on the sphere are the
fixed lines of $S$. In particular, the corresponding line subbundle $L_S\subset V$
of the embedded round sphere $S$ satisfies $SL_S = L_S$.  Hence $S$ induces complex structures
$J$ on $V/L_S$ and $\widetilde J$ on $L_S$ and the conformality equation
of the  sphere $S$ is
\[
*\delta_S = S\delta_S = \delta_S S\,.
\]
A sphere congruence assigns to each point $p\in M$ an oriented round
sphere $S(p)$ in $S^4$. In other words, a sphere congruence is a
complex structure $S\in\Gamma(\End(V))$ on the trivial
$\H^2$--bundle $V$.  Given a conformal immersion $f: M \to S^4$ with
associated line bundle $L=f^*\T$, a sphere congruence $S$ \emph{envelopes} $f$ if
for all $p\in M$ the sphere $S(p)$ passes through $f(p)$ and the
oriented tangent plane to $f$ and to $S(p)$ at $f(p)$ coincide:
\begin{equation}
\label{eq:enveloping}
S L = L, \quad \text{ and   }  \quad *\delta = S\delta = \delta S\,.
\end{equation}
Note that $S$ induces the complex
structures $J = S_{V/L}$ and $\widetilde J = S|_L$ given by the
conformality \eqref{eq:conformality} of $f$.

Let $\omega\in\Omega^1(W)$ be a 1--form on $M$ with values in a vector
bundle $W$. If $W$ is equipped with a complex structure
$J\in\Gamma(\End(W))$, $J^2=-1$, we can decompose $\omega$ into its
$(1,0)$ and $(0,1)$--parts with respect to $J$, that is
\[
\omega = \omega' + \omega''
\]
where
\[
\omega' =\frac{1}{2}(\omega - J*\omega), \quad
\omega'' =\frac{1}{2}(\omega + J*\omega)\,.
\]
We denote by $\Gamma(KW)$ and $\Gamma(\bar K W)$ the $(1,0)$
respectively $(0,1)$--forms with values in the complex bundle $(W,
J)$. For instance, if $f: M \to S^4$ is a conformal immersion its
derivative $\delta\in\Gamma(K\Hom(L,V/L))$ is a $(1,0)$--form by
(\ref{eq:conformality}).

A conformal immersion $f: M \to S^4$ induces \cite{conformal_tori} an
elliptic first order differential operator
\[
D: \Gamma(V/L) \to \Gamma(\bar K V/L),
\]
 a so--called \emph{quaternionic
  holomorphic structure} on $V/L$ given by
\begin{equation}
\label{eq:holomorphic structure}
D\varphi = (\pi d\widetilde\varphi)''\,.
\end{equation}
Here $\widetilde\varphi\in\Gamma(V)$ is an arbitrary lift of
$\varphi\in\Gamma(V/L)$ under $\pi$.  The holomorphic structure $D$ is
well--defined since $\pi d|_L= \delta\in \Gamma(K\Hom(L, V/L))$ and
thus $ D\psi= (\delta\psi)''=0 $ for $\psi\in\Gamma(L)$. We denote by
\[
H^0(V/L) = \ker D
\]
the space of \emph{holomorphic sections} of $V/L$.

An important property of the holomorphic structure $D$ is that $f$ is
given as a quotient of holomorphic sections and Darboux transforms of
$f$ are given by prolongations of holomorphic sections.  If $W\to M$
is a quaternionic line bundle over $M$ we write $\widetilde{W}$ for
its pull--back to the universal cover $\widetilde{M}$. A section
$\varphi\in\Gamma(\widetilde{W})$ with monodromy is one which
satisfies
\[
\gamma^*\varphi = \varphi h_\gamma, \quad
h_\gamma\in\H_{*}
 \]
for a representation $h$ of the fundamental group
$\pi_1(M)$ acting by deck transformations on $\widetilde{M}$.

\begin{lemma}[\cite{conformal_tori}]
\label{lem:prolongation}
Let $f\colon M \to S^4$ be a conformal immersion, $L \subset V$
the associated line subbundle of $V$ and $\varphi\in
H^0(\widetilde{V/L})$ a non--trivial holomorphic section. Then
$\varphi$ has a unique lift, with respect to the projection
$\pi\colon V\to V/L$, to a section
$\hat\varphi\in\Gamma(\widetilde{V})$ such that
\begin {equation}
    \pi d \hat\varphi=0\,.\label {eq:prolongation}
\end {equation} 
Away from its (isolated) zeros the \emph{prolongation} $\hat\varphi$
of $\varphi$ defines a conformal map $\hat{f}\colon \widetilde{M} \to
S^4$, namely $\hat{f}(p)=\hat\varphi(p)\H$. If the holomorphic section
$\varphi\in H^0(\widetilde{V/L})$ has mono\-dromy then $\hat{\varphi}$ has the same monodromy and thus
$\hat{f}$ descends to a conformal map $\hat{f}\colon {M} \to S^4$.
\end{lemma}

\begin{definition} Let $f\colon M \to S^4$ be a conformal immersion and
  $L \subset V$ be the associated line subbundle of $V$.  Conformal maps
  $\hat{f}$ defined (away from isolated points) by holomorphic
  sections of $\widetilde{V/L}$ are called {\em Darboux transforms} of
  $f$.
\end{definition}

Darboux transforms naturally generalise the classical Darboux
transforms since $\hat{f}$ is a Darboux transform of $f$ if and only
if (away from isolated points) there exists a sphere congruence $S$
enveloping $f$ and left--enveloping $\hat{f}$ \cite {conformal_tori}.
To say that $S$ \emph{left--envelopes} $\hat{f}$ means that
 \begin {equation}
  \label {eq:left--touch}
  S\hat{L}=\hat{L}\quad\text{and}\quad *\hat{\delta}=S\hat{\delta}\,,
\end {equation}
the latter expressing only half of the enveloping condition
\eqref{eq:enveloping}.

\begin{definition} Let $f\colon M\to S^4$ be a conformal immersion. A
  conformal map $f^\sharp: M \to S^4$ is called a \emph{classical
    Darboux transform} of $f$ if $f(p)\not=f^\sharp(p)$ for all $p\in
  M$ and if there exists a sphere congruence enveloping both $f$ and
  $f^\sharp$. In this case $(f, f^\sharp)$ is called a \emph{classical
    Darboux pair}.
\end{definition}
Darboux \cite{darboux} showed that $(f, f^\sharp)$ form a classical
Darboux pair if and only if $f$ and $f^\sharp$ are both
\emph{isothermic}, that is $f$ and $f^\sharp$ allow conformal
curvature line parametrisations away from umbilic points.

We now express the condition for $(f, f^\sharp)$ to be a classical
Darboux pair in terms of the derivatives $\delta$ and $\delta^\sharp$
of $f$ and $f^\sharp$. Given two surfaces $f, f^\sharp: M \to S^4$
with $f(p)\not=f^\sharp(p)$ for all $p\in M$, the trivial
$\H^2$--bundle $V$ splits as
\[
V = L\oplus L^\sharp\,.
\]
We write the trivial connection $d$ on $V$ in this
splitting as
\[
d = \begin{pmatrix}  \nabla^L & \delta^\sharp \\
\delta & \nabla^\sharp
\end{pmatrix}
\,,
\]
where $\nabla^L$ and $\nabla^\sharp$ are connections on $L$ and
$L^\sharp$ respectively. Moreover,
\[
\delta\in\Omega^1(\Hom(L,L^\sharp)) \quad
 \text{and}\quad  \delta^\sharp\in\Omega^1(\Hom(L^\sharp,L))\,,
 \]
 are the derivatives of $f$ and $f^\sharp$ when we identify $V/L =
 L^\sharp$ and $V/L^\sharp= L$ via the bundle isomorphisms
 $\pi|_{L^\sharp}: L^\sharp\to V/L$ and $\pi^\sharp|_L:L\to
 V/L^\sharp$.
 \begin{lemma}[\cite{coimbra}]
   \label{lem:dt with derivatives} Let $f, f^\sharp: M \to S^4$ be
   conformal immersions. Then $(f,f^\sharp)$ is a classical Darboux
   pair if and only if $L\oplus L^\sharp = V$, and
\begin{equation}
\label{eq:DT}
\delta^\sharp \wedge \delta = \delta\wedge \delta^\sharp = 0\,.
\end{equation}
\end{lemma}
\begin{proof} Let $(f,f^\sharp)$ be a classical Darboux pair with
  enveloping sphere congruence $S$.  Then
\[
*\delta=S\delta=\delta S \quad \text{ and }\quad *\delta^\sharp=S\delta^\sharp=\delta^\sharp S
\]
say that $\delta$ and $\delta^{\sharp}$ have type $(1,0)$ with respect
to the complex structure $S$ and so
\[
\delta\wedge \delta^\sharp=\delta^\sharp\wedge \delta=0\,.
\]
Conversely, the conformality equations
\[
*\delta = J \delta = \delta \widetilde J \quad \text{ and } *\delta^\sharp
= J^\sharp \delta^\sharp = \delta^\sharp \widetilde J^\sharp \,
\]
for $f$ and $f^\sharp$   together with (\ref{eq:DT}) give  $\widetilde J
= J^\sharp$ and $\widetilde J^\sharp = J$.  Hence the complex structure
\[
S = \begin{pmatrix} J^{\sharp} & 0\\
0 & J
\end{pmatrix}
\]
expressed in
the splitting $V = L \oplus L^\sharp$ envelopes $f$ and
$f^\sharp$ and therefore $(f,f^\sharp)$ form a classical Darboux pair.
\end{proof}

So far our considerations have been M\"obius invariant. Choosing a
point at infinity $\infty\in S^4$, lying neither on $f$ nor
$f^\sharp$, we can consider $f, f^\sharp\colon M \to\H$ as maps into
Euclidean 4--space. Then we can write $f^\sharp = f+T$ with $T: M \to
\H_*$ provided that for all $p\in M$, we have $f(p)\not= f^\sharp(p)$.
We may take $\infty = e\H$ with $ e
=\begin{pmatrix}1\\
 0
\end{pmatrix}\in\H^2
$
so that
\begin{equation}
\label{eq:psi}
\psi=\begin{pmatrix} f \\
1
\end{pmatrix}\in\Gamma(L), \quad \psi^\sharp=\begin{pmatrix} f^\sharp \\
1
\end{pmatrix}
\in\Gamma(L^\sharp)
\end{equation}
give trivialisations of $L$ and $L^\sharp$ respectively.  The
derivatives of $f$ and $f^\sharp$ in the splitting $V = L
\oplus L^\sharp$ then calculate to
\begin{equation}
\label{eq:delta in coordinates}
\delta\psi= \mbox {pr}_{L^\sharp}\begin{pmatrix} df \\
0
\end{pmatrix}
=\psi^\sharp T\invers df \quad \text{ and } \quad  \delta^\sharp\psi^\sharp= \mbox {pr}_{L}\begin{pmatrix} df^\sharp \\
0
\end{pmatrix}
=-\psi T\invers df^\sharp\,.
\end{equation}
If $f$ and $f^\sharp$ are a classical Darboux pair, then (\ref{eq:DT}) yields
\begin{equation}
\label{eq:isothermic}
 T\invers df^\sharp T\invers \wedge df =df \wedge T\invers df^\sharp T\invers=0\,,
\end{equation}
and thus
\[
d(T\invers df^\sharp T\invers)=
-T\invers d T T\invers \wedge df^\sharp
 T\invers
+ T\invers df^\sharp\wedge T\invers d T T\invers=
0\,,
\]
where we used $df^\sharp = df + dT$.  Therefore, locally $T\invers
f^\sharp T\invers= f^d$ for a conformal immersion $f^d: M\to\R^4$
satifisfying
\begin{equation}
\label{eq:dual}
df\wedge df^d = df^d \wedge df =0\,.
\end{equation}
A conformal map $f^d\colon M \to\R^4$ satisfying (\ref{eq:dual}) is
called a \emph{dual surface} to $f$. If a dual surface exists it is
unique up to translation and a real scaling \cite{udos_habil}.  As we
have seen any isothermic surface $f\colon M \to\R^4$ admits a dual surface
$f^d$.  The converse also holds as can be seen by reversing the above calculations: we obtain classical Darboux transforms
$f^\sharp$ by solving the Riccati equation
\begin{equation}
\label{eq: riccati}
dT = -df + Tdf^d T
\end{equation}
and putting $f^\sharp = f+ T: M \to\R^4$.  This is a well--known
description of isothermic surfaces via their dual surfaces
\cite{darboux_isothermic}.

\section{Constant mean curvature surfaces}
\label{sec:cmc}
We now turn to the case when the immersion $f\colon M\to\R^3$ has
constant mean curvature. Then $f$ is isothermic and applying the
classical Darboux transformation we obtain isothermic surfaces. These
have constant mean curvature when the initial condition for the
Riccati equation (\ref{eq: riccati}) is chosen appropriately
\cite{darboux_isothermic}. On the other hand, a surface of constant
mean curvature has an associated $\C_*$--family of flat connections
$\nabla^\mu$. We show that the parallel sections of these connections
give rise to Darboux transforms of $f$ which we call
\emph{$\mu$--Darboux transforms}. Only for special values of the spectral
parameter $\mu$ do they become classical Darboux transforms.

We view $\R^3=\Im\H$ as the imaginary quaternions. Then the Gauss map $N:
M\to S^2\subset\R^3$ of a conformal immersion $f: M\to\R^3$ satisfies
\[
*df  = N df = -df N\
\]
and $N$ is harmonic if and only if $f$ has constant mean curvature
\cite{RV:70}. The harmonicity condition for $N: M \to S^2 \subset\R^3$
is given by
\begin{equation}
\label{eq:Gauss map harmonic}
d(dN)''=0
\end{equation}
where $(dN)'' = \frac{1}{2}(dN + N*dN)$ is the (0,1)--part of $dN$
with respect to $N$. Note that $N$ is a complex structure since
$N^2=-1$. The splitting of $dN$
\[
dN = (dN)' + (dN)''
\]
into $(1,0)$ and $(0,1)$--parts is \cite[p.40]{coimbra} the
decomposition of the shape operator into trace and tracefree parts so
that
\begin{equation}
\label{eq:def H}
(dN)' = - H df,
\end{equation}
where $H$ is the mean curvature of $f$. With this normalisation the
mean curvature of the unit sphere with respect to the inward normal is
$H=1$. From now on we assume that we have scaled our constant mean curvature
surfaces so that  $H=1$.  Then the parallel
constant mean curvature surface
\[
g=f+N
\]
satisfies $ dg = (dN)''$ and thus type considerations give $
dg\wedge df = df \wedge dg =0\,.$ In other words, if $f: M \to\R^3$
has constant mean curvature then the parallel surface $g=f+N$ is a
dual surface of $f$ which shows that $f$ is isothermic.  We consider
$f: M \to\R^3$ as a conformal immersion into $S^4$ via
\[
\begin{pmatrix} f\\
1
\end{pmatrix}\H: M \to \HP^1\,,
\]
where the point at infinity is $e\H$ with $e =\begin{pmatrix} 1\\0
\end{pmatrix}\in\H^2$. The conformality of $f$ gives complex structures $J$ on
$V/L$ and $\widetilde J$ on $L$ satisfying (\ref{eq:conformality}).
Identifying $V/L$ with $\Li$ via the splitting $V= L\oplus \Li$ these
complex structures are given by
\begin{equation}
\label{eq:J}
J e =  e N \quad \text{ and } \quad \widetilde J\psi = -\psi N
\end{equation}
for the trivialising section $\psi=\begin{pmatrix} f\\1
\end{pmatrix}$ of $L$. In particular, for a constant mean curvature
surface the complex structure $J\in\Gamma(\End(V/L))$ is harmonic. Let
$\nabla$ denote the trivial connection on $V/L = \Li$ induced by the
trivial connection $d$ on $V = \underline{\H}^2$. From (\ref{eq:J})
and (\ref{eq:Gauss map harmonic}) we see that the harmonicity equation
for $J$ is
 \[
d^\nabla(\nabla J)'' = 0\,.
\]
With the  notation
\[
(\nabla J)' = -2*A \quad \text{ and } \quad  (\nabla J)'' = 2*Q
\]
 for the $(1,0)$ and $(0,1)$ parts, the harmonicity of $J$ becomes
\begin{equation}
\label{eq:harmonicity}
d^\nabla *A =0 \quad \text{ or equivalently} \quad d^\nabla*Q =0\,.
\end{equation}
Since $J^2=-1$ we see that $\nabla J$ anticommutes with $J$ and therefore
\[
*A = JA = - AJ \quad \text{ and } \quad *Q = -JQ = QJ\,.
\]
To reformulate the harmonicity of $J$ as a $\C_*$--family of flat
${\bf SL}(2,\C)$-con\-nec\-tions we introduce the constant complex
structure $I$, which is defined as right multiplication $I\varphi =
\varphi i$ by the quaternion $i$. With this complex structure $V/L
=\underline{\C}^2$ can be viewed as a trivial $\C^2$-bundle. The next
lemma is a variant \cite{klassiker} of the well-known formulation of
harmonicity in terms of families of flat connections.

\begin{lemma}
\label{lem:family_unit_circle}
Let $J\in\Gamma(\End(V/L))$ be a complex structure on $V/L$ with flat
connection $\nabla$. Then $J$ is harmonic if and only if the complex
connections
\begin{equation*}
\nabla^\mu = \nabla+ *A(J\frac{\mu+\mu\invers-2}{2} + \frac{\mu\invers-\mu}{2}I)
\end{equation*}
on the complex bundle $(V/L, I)$ are flat for all $\mu = u+ Iv\not=0$,
$u,v\in\R$.
\end{lemma}
\begin{rem}
\label{rem:lambda}
There are a number of useful ways to rewrite the family of flat connections
$\nabla^\mu$.  If we put
\begin{equation*}
a = \frac{\mu+\mu\invers}{2}, \quad b =\frac{\mu\invers-\mu}{2}I
\end{equation*}
then $a^2+b^2=1$ and
\begin{equation}
\label{eq:complex_family}
\nabla^\mu = \nabla+ *A(J(a-1)+b) = \nabla+(a -1 + Jb)A
\,,
\end{equation}
where we used $*A = JA =-AJ$ and $[A,I]=0$.  On the other hand, using
the type decomposition
\[
A^{(1,0)} =\frac{1}{2}(A - I*A), \quad A^{(0,1)} =\frac{1}{2}(A+I*A)
\]
of $A$  with respect to the complex structure $I$, we obtain
\begin{equation}
\label{eq:nablaIdec}
\nabla^\mu = \nabla + (\mu-1) A^{(1,0)} + (\mu\invers -1)A^{(0,1)}\,.
\end{equation}
Note that for $\mu\in S^1$, that is $a, b\in\R$, the connection $
\nabla^\mu $ is in fact quaternionic whereas $\nabla^\mu$ is a complex
connection for $\mu\not\in S^1$ since the complex structure $I$ is not
quaternionic linear. Moreover, we see from \eqref{eq:nablaIdec} that
\begin {equation}
(\nabla ^\mu\phi)j =\nabla ^ {\bar\mu^ {- 1}}(\phi j)
\label {eq:reality}
\end {equation}
for $\phi\in\Gamma(V/L)$.
\end{rem}
\begin{proof}[Proof of Lemma \ref{lem:family_unit_circle}]
We have $\nabla J = 2(*Q-*A)$, and by type
  considerations we see that $A\wedge Q =0$, so
 \[
d^\nabla(*AJ) = (d^\nabla*A)J - *A \wedge \nabla J = (d^\nabla*A)J +2 *A \wedge *A\,.
\]
{F}rom this the curvature of $\nabla^\mu$ computes to
\[
R^\mu = (d^\nabla*A)(J(a-1) +b)
\]
where we used $[A,I]=0$ and $a^2 + b^2=1$.  This shows that
$\nabla^\mu$ is a flat  connection for every $\mu\in \C_*$ precisely when
$d^\nabla*A=0$, that is, if and only if $J$ is harmonic.
\end{proof}

The parallel sections of $\nabla^\mu$ for $\mu\in\C_*$ can be given a
geometric interpretation in terms of Darboux transforms of $f$. One
observes from \eqref{eq:holomorphic structure} that the trivial
connection $\nabla$ is compatible with the holomorphic structure $D$
on $V/L$, that is $\nabla''=D$. Since $*A=JA$, equation
\eqref{eq:complex_family} then shows that also
 \[
 (\nabla^{\mu})''=D\,.
\]
Hence $\nabla^{\mu}$--parallel sections of $V/L$ are in particular
holomorphic without zeros and their prolongations give Darboux
transforms $\hat f$ of $f$ defined on all of $\widetilde M$.  Since
$f(p)\neq \hat{f}(p)$ for all $p\in \widetilde{M}$ we have the
splitting $L\oplus \hat L = V$.

\begin{definition}
  Let $f\colon M \to\R^3$ be a constant mean curvature surface. The
  Darboux transforms $\hat f\colon \widetilde M \to S^4$ given by
  sections of $\widetilde{V/L}$ that are parallel with respect to
  $\nabla^\mu$ for $\mu\in\C_*$ are called {\em $\mu$--Darboux
    transforms} of $f$.
\end{definition}

\begin{figure}[h]
\begin{center}
\includegraphics[width=0.45\linewidth]{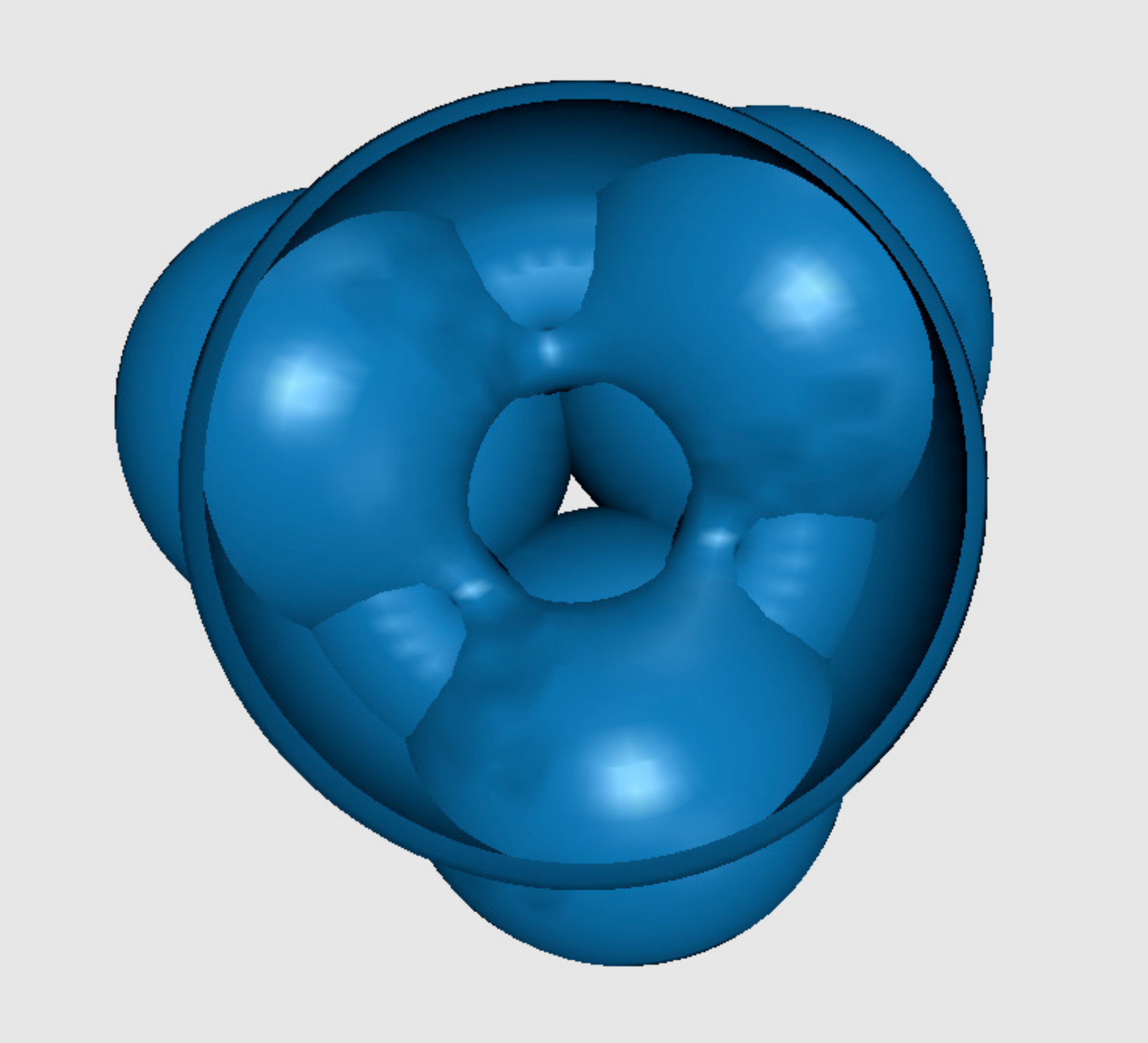}
\includegraphics[width=0.45\linewidth]{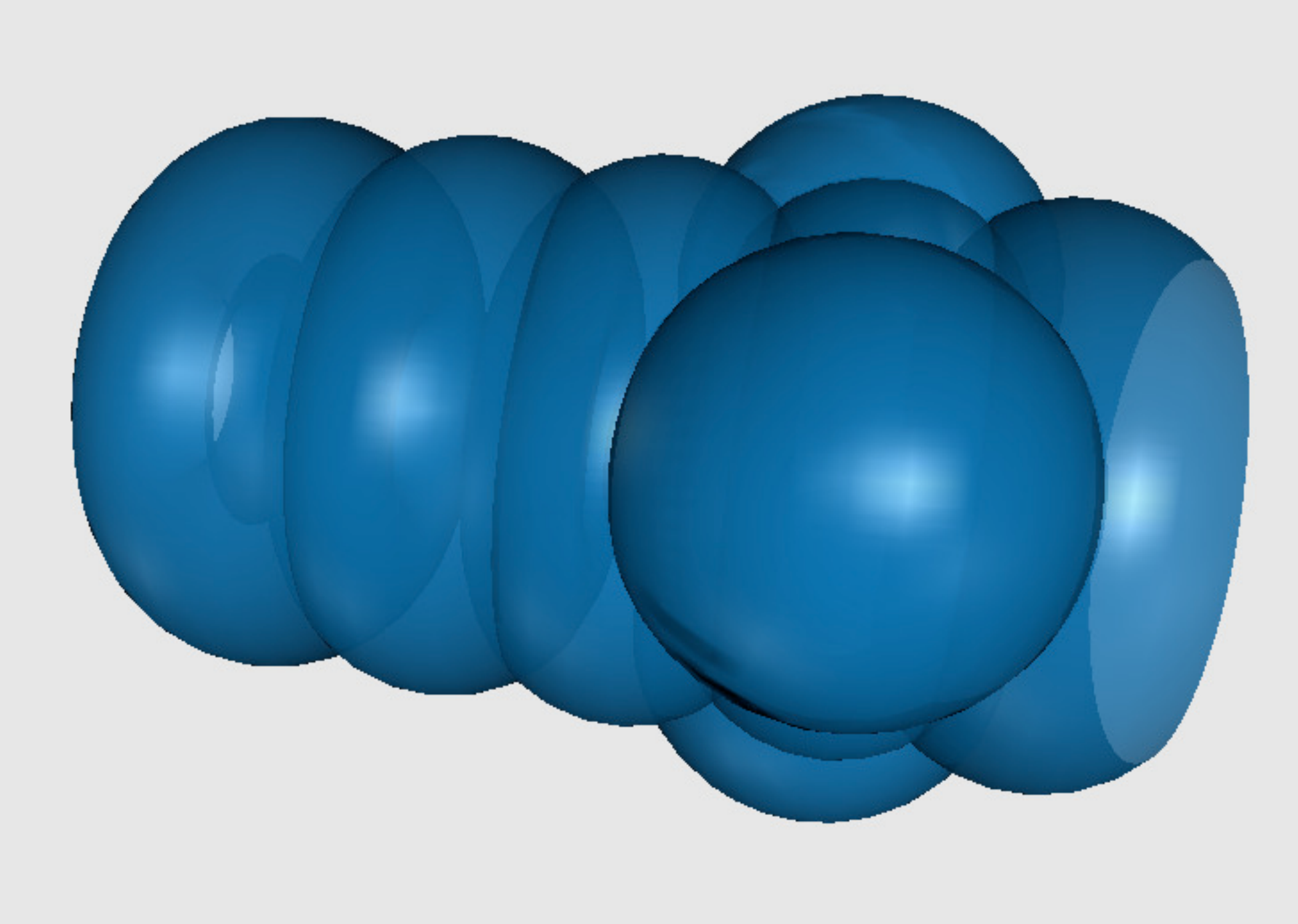}
\caption{
Closed $\mu$--Darboux transform of a nodoid.}
\label{fig:2Bubbleton}
\end{center}
\end{figure}

If $M$ has topology it is generally difficult to decide whether a
constant mean curvature surface $f$ has a $\mu$--Darboux transform
$\hat f$ defined on $M$ rather than on its universal covering $\widetilde
M$.  If $M$ is a 2--torus this question leads to the notion of the
spectral curve which we will address in section 4. For now we are only
interested in the local properties of $\mu$--Darboux transforms $\hat
f$ and assume that $M$ is simply connected. Then there is a
$\CP^1$--worth of $\nabla^\mu$--parallel sections for each
$\mu\in\C_*$ and thus the space of $\mu$--Darboux transforms is parameterised by $\C_{*}\times\CP^1$.
\begin{figure}[h]
\begin{center}
\includegraphics[width=0.45\linewidth]{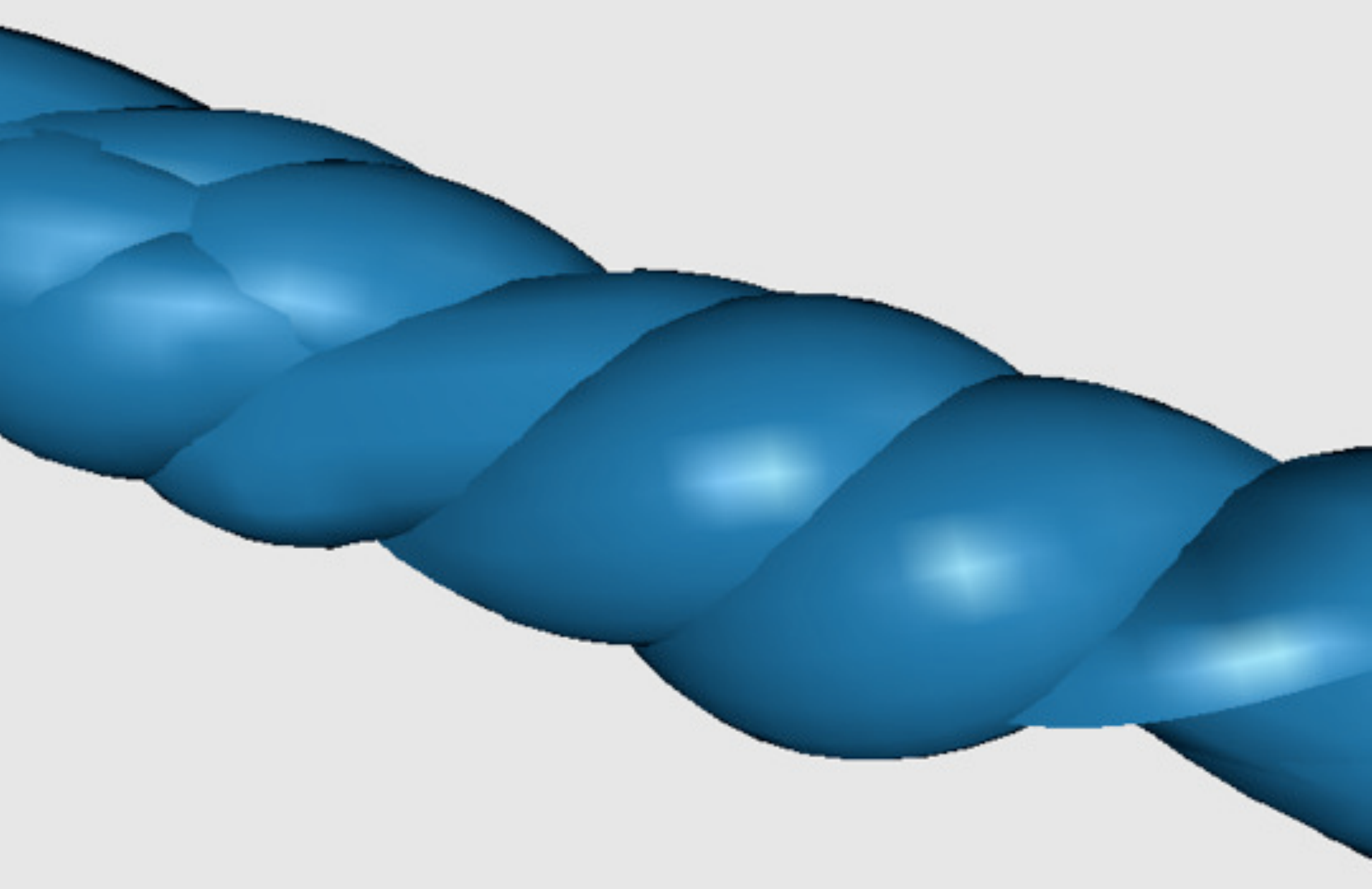}
\caption{Non-closed $\mu$--Darboux transform of an unduloid.}
\end{center}
\end{figure}

To better understand the geometry of $\mu$-Darboux transforms we
reinterpret the prolongation in terms of a bundle homomorphism, the
\emph{prolongation map}.

Let $\varphi\in\Gamma(V/L)$ be a $\nabla^\mu$--parallel section. The
splitting $V = L \oplus \Li$ identifies $V/L =\Li$ so that the
prolongation
\begin{equation}
\hat\varphi = \varphi + \hat B\varphi
\end{equation}
of the nowhere vanishing section $\varphi$ defines an $\H$--linear
bundle map
\begin{equation*}
\hat B: \Li\to L\,.
\end{equation*}
From
$
0=\pi \nabla \hat\varphi$ and $\nabla^\mu\varphi=0$  we deduce
\begin{equation}
\label{eq:Bmuc}
\delta \hat B\varphi = *A(J(a-1) + b)\varphi\,.
\end{equation}
Here $a=\frac{\mu+\mu\invers}2, b=\frac{\mu\invers-\mu}2 I$ for
$\mu\in\C_*$ with respect to the constant complex structure $I$ on
$V/L$ given by right multiplication by the quaternion $i$. We denote
by $\hat I$ the quaternionic linear complex structure on $V/L$ given
by the quaternionic linear extension of $I$ on $\varphi$, that is, $\hat I
\varphi = \varphi i$. Furthermore,  $\hat a, \hat b\in\Gamma(\End(V/L)$ denote
the quaternionic linear bundle maps obtained from $a, b$ by replacing
$I$ with $\hat I$.  Then
\begin{equation}
\label{eq: delta hat B}
\delta \hat B = *A(J(\hat a-1) + \hat b)\in\Gamma(\End(V/L))
\end{equation}
and
\begin{equation}
\label{eq:quater conn}
\hat\nabla^\mu = \nabla +*A(J(\hat a-1)+\hat b)
\end{equation}
is a family of flat quaternionic connections on $V/L$ with
\begin{equation}\label{eq:hat I parallel}
\hat\nabla^\mu\hat I = 0\,.
\end{equation}
Note that $\hat I$ and therefore $\hat\nabla^\mu$ depend on the choice
of the $\nabla^\mu$--parallel section $\varphi$. We indicate this
dependence by decorating $\varphi$ dependent quantities by the ``hat''
symbol. In what follows we abbreviate
\begin{equation}
\label{eq:C}
\hat C = J(\hat a-1) + \hat b
\end{equation}
and thus (\ref{eq: delta hat B}) becomes
\begin{equation}
\label{eq: delta B = AC}
\delta\hat B = *A \hat C\,.
\end{equation}

In order to see that a $\mu$--Darboux transform $\hat f$ of a constant
mean curvature surface $f$ is up to translation in $\R ^ 4 $ a constant mean
curvature surface in $\R^3$, we need to compare geometric data of $f$
and $\hat f$ with respect to the same choice of point at infinity
$\infty=e\H$. For that we need to see that $\hat f$ does not pass
through $\infty$.  Since $f$ is an immersion its derivative $\delta$
has no zeros, so as $\delta$ and $*A$ have the same type,
\begin{equation}
\label{eq:A=delta}
*A = \delta R
\end{equation}
for $R\in\Gamma(\Hom(V/L, L))$. Evaluating both sides of
\eqref{eq:A=delta} on the constant section $e$ we obtain $Re =
\frac{1}{2}\psi$ where we recall (\ref{eq:def H}), (\ref{eq:J}) and
the trivialising section $\psi\in\Gamma(L)$ from \eqref{eq:psi}. This
shows that $R$ is nowhere vanishing and parallel with respect to the
connection induced by $\nabla$ on $V/L$ and $\nabla^{L}$ on $L$.  Now
$\hat f$ passes through $\infty$ if and only if the prolongation
$\hat\varphi = \varphi+\hat B\varphi$ of the $\nabla^\mu$ parallel
section $\varphi\in\Gamma(V/L)$ giving rise to $\hat f$ has
$\hat\varphi(p)\in e\H$ and thus $\hat B(p)=0$ for some $p\in M$. But
$\delta$ and thus by (\ref{eq:A=delta}) 
also $*A$ have no zeros
by the assumption that $f$ is immersed. Therefore (\ref{eq: delta B =
  AC}) shows that $\hat C= J(\hat a-1) +\hat b$ has to vanish at $p$
which implies $\mu=1$. In this case $\hat\nabla^\mu = \nabla$, the
prolongation $\hat \varphi$ is constant and $\hat f$ is the point
$\infty$. In the considerations to follow, we always assume
$\mu\not=1$, and in particular, both $\hat B$ and $\hat C$ are nowhere
vanishing.

\begin{lemma}
\label{lem: mu gives DT}
Let $f: M \to\R^3$ be a conformal immersion of constant mean curvature
with $*\delta = J\delta = \delta\widetilde J$ and $\hat f: M \to S^4$ a
non--constant $\mu$--Darboux transform of $f$. Then the derivative of
$\hat f$, expressed in the splitting in $V = \hat L \oplus \Li$, is
given by
\begin{equation}
\label{eq:hatdelta=hata}
\hat \delta \hat R= *\hat  A\,.
\end{equation}
Here $\hat R = \hat C \invers + R$ and $-2\hat A = \frac 12(\nabla\hat
J - \hat J *\nabla\hat J)$ is the $(1,0)$--part of the derivative of
the complex structure
\[
\hat J = -\hat C\invers J \hat C\,.
\]
Moreover, $\hat J$ is harmonic and $\hat f$ is conformal with $*\hat
\delta =\hat J \hat \delta =\hat \delta \widetilde{\hat J}$ where
$
 \widetilde{\hat J} = -\hat R \hat J
\hat R\invers\,.
$
\end{lemma}
\begin{proof}
  Let $\varphi$ be a $\nabla^\mu$ parallel section, $\hat\varphi =
  \varphi + \hat B\varphi$ its prolongation, and $\hat L
  =\hat\varphi\H$ the corresponding $\mu$--Darboux transform. Using
 (\ref{eq: delta B = AC}) and the decomposition
\[
d = \begin{pmatrix} \nabla^L & 0 \\ \delta & \nabla
\end{pmatrix}
\]
of $d$ in the splitting $V = L \oplus\Li$ we get
\begin{equation*}
d\hat\varphi = \nabla\varphi +
\nabla^L(\hat B\varphi) + \delta(\hat B\varphi)= (\nabla \hat
B)\varphi - \hat B\delta \hat B\varphi\,.
\end{equation*}
For $\phi\in L$ we have
\[
\phi =  (\phi +\hat B\invers\phi) - \hat B\invers\phi \in \hat L \oplus \Li\,,
\]
which shows that $-\hat B\invers\in\Hom(L, \Li)$ is the projection of
$L$ onto $\Li$.  Calculating $\hat\delta\hat\varphi =\pi_{\hat L}
d\hat\varphi$ we obtain with $d\hat\varphi\in\Omega^1(L)$ that
\begin{equation}
\label{eq:delta with B}
\hat\delta =  (-\hat B\invers (\nabla\hat B) +\delta\hat B)(1+\hat B)\invers\,,
\end{equation}
where we also used $\hat\varphi = (1+\hat B)\varphi$. Recalling
(\ref{eq: delta B = AC}) and (\ref{eq:A=delta}) and the fact that $R$ is parallel, we get
\begin{equation}
\label{eq:hilf}
-\hat B\invers(\nabla \hat B) + \delta \hat B = -\hat C\invers\nabla
\hat C + *A\hat C\,.
\end{equation}
Furthermore, since $\hat\nabla^\mu\hat I =0$ we have $ \nabla\hat I =
-[*A\hat C, \hat I]$ and
\begin{eqnarray*}
\nabla \hat C &=&  2*Q(\hat a-1) + \hat C *A \hat C +  *A(-2(\hat a-1) + J\hat C(\hat a-1) -\hat C\hat b)\,.
\end{eqnarray*}
But since
 $\hat
a^2 + \hat b^2=1$ and $\hat C = J(\hat a-1) +\hat b$ we have $-2(\hat a-1) + J\hat C(\hat a-1) -\hat C\hat b =0$, or equivalently
\begin{equation}
\label{eq: initial}
\hat J = \frac{\hat b}{1-\hat a} - 2\hat C\invers\,.
\end{equation}
Thus we conclude that
\begin{equation}
\label{eq:Riccati for C}
\nabla \hat C = 2*Q(\hat a-1) + \hat C*A\hat C\,.
\end{equation}
This Riccati type equation together with (\ref{eq:delta with B}) and (\ref{eq:hilf}) yields
\begin{equation}
\label{eq:hatdelta}
\hat \delta = - 2\hat C\invers *Q(\hat a-1)(1+\hat B)\invers\,,
\end{equation}
and, since $*Q = - J *Q$, we also have $*\hat \delta = \hat J \hat
\delta$ with $\hat J = - \hat C\invers J \hat C$.\\

On the other hand, the derivative of $\hat J$
computes to
\[
\nabla \hat J =  \hat C\invers\left(2*Q\big(J(\hat a-1)-(\hat a-1)\hat J\big) - \big(\hat C J + J\hat C)*A\hat C - (\nabla J) \hat C\right)\,,
\]
so that its $(1,0)$--part with respect to $\hat J$ is given by
\begin{equation*}
(\nabla \hat J)' =- 2 \hat C\invers*Q\left((\hat a-1)\hat J + \hat b\right)\,.
\end{equation*}
Using (\ref{eq: initial})
we thus obtain
\begin{equation}
\label{eq:hatA}
*\hat A = -\frac 12(\nabla \hat J)' =   -2\hat C\invers*Q(\hat a-1)\hat C\invers\,.
\end{equation}
Comparison with (\ref{eq:hatdelta}) gives together with (\ref{eq: delta B =
  AC}) and (\ref{eq:A=delta})
\begin{equation*}
\hat \delta \hat R = *\hat  A
\end{equation*}
where $\hat R = \hat C \invers + R$, and from $*\hat A = \hat J \hat A =
-\hat A \hat J$ we also see $*\hat\delta = \hat J \hat \delta =
\hat \delta\widetilde{\hat J}$.

Finally, by (\ref{eq:hatA}) and the Riccati type equation
(\ref{eq:Riccati for C}) we see that
\[
*\hat A  = *A + \nabla \hat C\invers
\]
is $d^{\nabla}$-closed since $d^\nabla*A=0$ by the harmonicity of $J$.

\end{proof}

Notice that equation (\ref{eq:hatdelta}) shows that a
$\mu$--Darboux transform is constant if and only if $Q= 0$ or
$\hat a =1$. The first condition means that the Gauss map $N$ is
holomorphic and thus $f(M)$ is contained in a round sphere whilst the
second is equivalent to $\mu=1$.

\begin{cor}
  If $f(M)$ is not contained in a round sphere, then a $\mu$-Darboux
  transform $\hat f$ of $f$ is constant if and only if $\mu=1$. In
  this case, $\hat f= \infty$.
\end{cor}

Using Lemma \ref{lem: mu
  gives DT} we arrive at our first main result that every
 $\mu$-Darboux transform of a constant mean curvature surface has
 constant mean curvature.

\begin{theorem} \label{thm:DT is CMC} Let
  $f\colon M \to\R^3$ be a surface of constant mean curvature $H=1$.
  Then every $\mu$-Darboux transform $\hat f$ of $f$ has constant real
  part and when $\hat{f}$ is not a point, $\Im(\hat f)$ is a constant
  mean curvature surface in $\R^3$ with $\hat H=1$. In particular, for
  each $\mu=e^{i\theta}\in S^1\setminus\{1\}$ there is a unique
  $\mu$-Darboux transform of $f$, given by $\hat f = g +
  \cot\frac{\theta}{2}$, where $g = f+N$ is the parallel constant mean
  curvature surface of $f$.
 \end{theorem}
   \begin{proof}
    Let $\varphi$ be a $\nabla^\mu$--parallel section and $\hat f$ the
    associated $\mu$--Darboux transform for $\mu\not=1$.  Writing
    $\hat C\invers e = \frac 12 e\hat T$ with $\hat T: M \to\H_*$ we have
\begin{equation}
\label{eq:hatR}
\hat R e  = (R + \hat C\invers) e
\ =\frac 12\begin{pmatrix} f + \hat T\\ 1
\end{pmatrix}
\in\Gamma(\hat L)
\end{equation}
so that $\hat f = f +\hat T$.  Moreover, the complex structures $\hat
J=-\hat C\invers J \hat C$ and $\widetilde{\hat J}=-\hat R \hat J \hat R
\invers$ satisfy
\begin{equation}
\label{eq:hat N}
\hat J e = e \hat N \quad \text{ and } \quad \widetilde{\hat J}
  \hat\psi = -\hat\psi \hat N
\end{equation}
with $\hat N = -\hat T N \hat T\invers$ and $\hat \psi = 2\hat R e$.  In
particular, $*\hat\delta\hat\psi = \hat J \hat\delta \hat \psi = \hat \delta \widetilde{\hat J}\hat \psi$  reads in coordinates as
\[
*d\hat f = \hat N d\hat f  = - d\hat f \hat N
\]
so that $d\hat f$ takes values in $\R^3$, and $\hat f$ has constant
real part. Therefore, $\hat N$ is the Gauss map of $\Im(\hat f)$. From
Lemma \ref{lem: mu gives DT} we know that $\hat J$ is harmonic and
thus $\Im(\hat f)$ has constant mean curvature.  In fact,
(\ref{eq:hatdelta=hata}) shows that the derivative of $\hat f$, and
thus of $\Im(\hat f)$, is $e\,d\hat f = -e(d\hat N)'$ so that the mean
curvature of $\Im(\hat f)$ is $\hat{H} =1$.  Writing
$\varphi=e\alpha$, then evaluating $\hat{C}$ on $e$ gives
\begin{equation}
\label{eq:Tinvers}
e \hat T\invers = \frac 12 \hat C e = \frac 12 e(N\alpha(a-1)\alpha\invers + \alpha b\alpha\invers)
\end{equation}
where $\mu$, $a =\frac{\mu+\mu\invers}2$ and $b =  \frac{\mu
  - \mu\invers}{2i}$ are viewed as complex numbers.  This shows for $\mu\in
S^1$, that is $a, b\in\R$, that the $\mu$-Darboux transform is given
by
\[
\hat{f} = f + N +\frac{b}{1-a}=g+\cot\tfrac{\theta}{2}
\]
and hence is a translate of the parallel constant mean curvature surface $g=f+N$ and is independent of the parallel section $\varphi$.
\end{proof}
\begin{rem}
\label{rem:realpart0}
 Observe that \eqref{eq:Tinvers} also shows that a $\mu$--Darboux
 transform $\hat f$ of $f$ has vanishing real part for
  $\mu\in\R_*$. Furthermore, for $\mu\in \C_*\setminus (S^1\cup \R_*)$
  the real part of a $\mu$--Darboux transform is $ \Re(\hat
  f) = \frac{\Im a}{\Im\left((a-1){\bar b}^{-1}\right)}$ where
  $a=\frac{\mu+\mu\invers}2$ and $b= \frac{\mu\invers-\mu}{2}I$.
\end{rem}

\begin{rem}
  There are versions of Lemma \ref{lem: mu gives DT} for other classes
  of immersions given by a harmonicity condition. Analogues of Theorem
  \ref{thm:DT is CMC} hold for Hamiltonian stationary Lagrangian
  surfaces \cite{hsl}, and for (constrained) Willmore surfaces in the
  4--sphere \cite{Bohle:08}.
\end{rem}

We conclude this section by determining which $\mu$--Darboux
transforms $\hat{f}$ of a constant mean curvature immersion $f$ are
classical.  Since a $\mu$--Darboux transform $\hat{f}$ is a Darboux
transform there is a sphere congruence $S$ enveloping $f$ and
left--enveloping $\hat{f}$, hence satisfying \eqref{eq:enveloping} and
\eqref{eq:left--touch}. From arguments analogous to those in the proof
of Lemma~\ref{lem:dt with derivatives} (see also
\cite{conformal_tori}) it follows that these enveloping conditions are
equivalently described by $\delta\wedge \hat{\delta}=0$. Therefore, in
order to see which $\mu$--Darboux transforms are classical, we need by
Lemma \ref{lem:dt with derivatives} to investigate the condition
$\hat\delta\wedge\delta=0$ in the splitting $V= L \oplus\hat L$.

\begin{theorem}
\label{thm:DT is not CDT}
A non--constant $\mu$--Darboux transform  $\hat{f}\colon M\to \R^4$
of a constant mean curvature
surface $f:M \to\R^3$ is a classical Darboux transform of $f$ if and
only if $\mu\in\R_*\cup S^1\setminus\{1\}$.

\end{theorem}
Note that by Remark~\ref{rem:realpart0} the $\mu$-Darboux transform $\hat{f}$ takes values in a parallel
translated $\R^3\subset \R^4$ of distance $\cot\tfrac{\theta}{2}$ where $\mu=e^{i\theta}$.
In particular, $\hat{f}$ takes values in the same $\R^3$ as $f$ if and only if $\mu\in\R\setminus\{0,1\}$.
\begin{proof}
Let $\varphi\in\Gamma({V/L})$ be a parallel section of $\nabla^\mu$
and $\hat f$ the Darboux transform given by $\varphi$. From the
definition of the prolongation map $\hat{B}$ we see that the
decomposition of $\phi\in\Li$ with respect to the splitting $V = L
\oplus\hat L$ is
\begin{equation*}
\phi = (-\hat B \phi) +(1+\hat B)\phi \in L \oplus\hat L\,.
\end{equation*}
Therefore, using equations (\ref{eq:A=delta}) and (\ref{eq:hatdelta})
the derivatives of $f$ and $\hat f$ with respect to this splitting are
given by
\begin{equation*}
\delta R = (1+\hat B)*A \quad \text{ and } \quad \hat \delta (1+\hat B) =2 R*Q(\hat a-1)\,.
\end{equation*}
Since $\hat f$ is not constant $\hat\delta$ has only isolated zeros
and so does $Q$. Then
\[
2 R*Q(\hat a-1) \wedge *A=\hat \delta\wedge \delta R=0
\]
if and only if $*Q(\hat a-1) = Q(\hat a-1)J$ where we used
$*A=JA$. Since $Q$ has isolated zeros this last is equivalent to $[J,
\hat a-1]=0$, that is
\begin{equation}
\label{eq:cdt condition}
(\Im \hat a) [J,  \hat I]=0\,.
\end{equation}
However the two complex structures $J$ and $\hat I$ commute only if $\hat
I = \pm J$. By \eqref{eq:hat I parallel} and \eqref{eq:quater conn}
this implies
\[
0=\hat\nabla^\mu J = 2(*Q-*A) + [*A\hat C, J]\,.
\]
Considering the $(0,1)$--part, this yields
\[
*Q  =0\,,
\]
contradicting the assumption that $\hat f$ is not constant.
Therefore, \eqref{eq:cdt condition} is equivalent to $\Im\hat a =0$,
which shows that a non--constant $\mu$--Darboux transform is a
classical Darboux transform
if and only if $\mu\in\R_*\cup S^1\setminus\{1\}$.
\end{proof}

Since the real part of a $\mu$--Darboux transform $\hat f$
  of a constant mean curvature surface $f\colon M \to\R^3$ is constant only
  its imaginary part $\Im(\hat f)$ is geometrically
  relevant.

  \begin{theorem}
    Let $f\colon M\to\R^3$ be a constant mean curvature surface and
    $\hat{f}$ a $\mu$-Darboux transform. Then $\hat{f}$ is a classical
    Darboux transform if and only if $\Im(\hat{f})$ is.
  \end{theorem}
  \begin{proof}
    If $\hat{f}$ is a classical Darboux transform then by Theorem
    \ref{thm:DT is CMC} and Remark \ref{rem:realpart0} either
    $\hat{f}$ has vanishing real part or it is the parallel constant mean
    curvature surface of $f$ up to a real translation. Therefore
    $\Im\hat{f}$ is also a classical Darboux transform. On the other
    hand, let us assume that $\Im(\hat{f})$ is a classical Darboux
    transform. Then it follows from (\ref{eq:isothermic}),
    (\ref{eq:hat N}) and the fact that $\hat{f}$ and $\Im(\hat{f})$
    differ by a real constant that
  \[
  \hat{N}=-\hat TN\hat T^{-1}=-\Im(\hat T) N \,\Im(\hat T)^{-1}\,,
  \]
  where $\hat T=\hat{f}-f$. Assuming $\Re(\hat T)=\Re(\hat{f})\neq 0$
  we conclude $\Im(\hat T)=rN$ with $r\colon M\to \R$ and thus
  $\hat{N}=-N$. From Theorem~\ref{thm:DT is CMC} we know that
  $\Im(\hat{f})$ is a constant mean curvature surface whose Gauss map
  is $-N$ and therefore $\Im(\hat{f})$ is the parallel constant mean
  curvature surface of $f$.  Now \eqref{eq: initial} and
  (\ref{eq:Tinvers}) imply that $\Re(\hat T) = \frac{\hat b}{1-\hat
    a}$ which, using $\hat a^2 +\hat b^2=1$, shows that $\hat a
  =\frac{\Re(\hat T)^2-1}{\Re(\hat T)^2 +1}\in\R$ and therefore
  $\mu\in S^1$. But then Theorem~\ref{thm:DT is not CDT} implies that
  $\hat{f}$ is a classical Darboux transform.
\end{proof}

So far we have seen that $\mu$--Darboux transforms are classical
Darboux transforms if and only if $\mu\in\R_*\cup S^1$. We now turn to
the question of which classical Darboux transforms are $\mu$--Darboux
transforms.

\begin{theorem}
  A classical Darboux transform $f^\sharp\colon M \to\R^4$ of a constant mean curvature surface
  $f\colon M\to \R^3$ is a
  $\mu$--Darboux transform if and only if $T =f^\sharp-f$ satisfies
  the Riccati equation
\begin{equation}
\label{eq:ini}
dT = r Tdg T -df\,, \quad
(T-N)^2 = r\invers-1\quad r\in\R\setminus\{0,1\}\,.
\end{equation}
In this case Theorem \ref{thm:DT is not CDT} shows that $\mu\in
\R_*\cup S^1\setminus\{1\}$.  When $\mu\in S^1$ the classical Darboux
transform $f^\sharp$ is a translate of the parallel constant mean curvature surface $g=f+N$
and for $\mu\in\R_{*}$ it is a constant mean curvature surface in $\R^3$.
\end{theorem}
\begin{proof}

  Let $f^\sharp: M \to\R^4$ be a $\mu$--Darboux transform for
  $\mu\in\R_*\cup S^1\setminus\{1\}$.  With (\ref{eq:Riccati for C})
  we see that $T = f^\sharp - f$ satisfies the Riccati equation with
  $r= \frac{1-\hat a}2\in\R$ since $\mu\in\R_*\cup S^1$. Now
  (\ref{eq:Tinvers}) gives
  \[
( (Tr)\invers + N)^2 = \frac{\hat b^2}{(1-\hat a)^2} = r\invers -1\,,
\]
where we used $\hat a^2 + \hat b^2=1$. Since $r\in\R_*$ this equation
is equivalent to $(T-N)^2 = r\invers -1$.

  Conversely, let $T$ be a solution of the Riccati equation with
  $r\in\R_*, r\not=1$ and (\ref{eq:ini}).  We put
\[
\hat a
= 1-2r  \quad \text{ and }  \quad\hat b = 2(T\invers + Nr) \,.
\]
If $r\in(0,1)$ then $\hat a^2 +\hat b^2=1$ and $\hat a\in\R$ imply
$\hat b\in\R$. In particular, $\mu=\hat a + I \hat b\in S^1$, and $T =
N + \frac{\hat b}{1- \hat a}$. Thus, $\hat f = f + T$ is a
$\mu$--Darboux transform of $f$ with $\mu\in S^1$.

If $r\in\R\setminus[0,1]$ then $\hat a^2 + \hat b^2=1$, $|\hat a|>1$,
and $\hat a\in\R$ imply that $\hat b$ has values in the imaginary
quaternions.  Moreover, the quaternionic connection $ \hat\nabla =
\nabla + \omega $, where $\omega\in\Omega^1(\Li)$ is defined by
\[
\omega e= edf T\invers\,,
\]
is flat since $T$ satisfies the Riccati equation. Let $\varphi =
e\alpha$ be a $\hat\nabla$--parallel section, that is
$(d\alpha)\alpha\invers = - df T\invers$. Using again the Riccati
equation together with (\ref{eq:ini}) we see
\[
d(\alpha\invers \hat b \alpha)=  2\alpha\invers\left( [df T\invers, T\invers + Nr]  + d(T\invers) + dN r\right)\alpha = 0\,.
\]
Thus we may assume after scaling $\varphi$ by a quaternion that $\hat
b\varphi = \varphi b_0 i$ with $b_0\in\R$ since $\hat b^2 < 0$. In
particular,
\[
(J(\hat a-1) + Ib_0)\varphi = 2\varphi T\invers
\]
so that $\varphi$ is a parallel section of the complex connection
$\nabla^\mu$ for $\mu= \hat a -b_0\in\R$.  The prolongation of
$\varphi$ gives a $\mu$--Darboux transform of $f$ which is exactly $f^\sharp
= f+T$.
\end{proof}

In \cite{darboux_isothermic} it is shown that a classical Darboux
transform $f^\sharp\colon M \to\R^3$ of a constant mean curvature surface
$f\colon M \to\R^3$ has constant mean curvature if and only if
$T=f^\sharp-f$ satisfies \eqref{eq:ini} with
$r\in\R\setminus\{0,1\}$. Therefore we obtain:
\begin{cor} Let $f: M\to\R^3$ be a constant mean curvature surface.
  The classical Darboux transforms $f^\sharp: M \to\R^3$  of $f$ with
  constant mean curvature are exactly the $\mu$--Darboux transforms of
  $f$ with $\mu\in\R\setminus\{0,1\}$.
\end{cor}

\section {The \eigenline Spectral Curve}

Thus far we have discussed the local properties of $\mu$--Darboux
transforms. We now turn our attention to the question of global
existence of solutions: given a constant mean curvature surface $f: M
\to\R^3$ where $M$ has topology, is there a $\mu$--Darboux transform
$\hat f: M\to S^4$ of $f$ defined on $M$ rather than on the universal
cover $\widetilde{M}$ of $M$? In general, this question is hard to
decide, however in the case of a constant mean curvature torus $f: T^2
\to \R ^ 3 $ we shall see that there are many such
transformations. In fact  we will show that the space of $\mu $--Darboux transforms of $ f $ is given by
Hitchin's spectral curve
\cite{Hitchin:90} together with finitely many complex projective lines.
Henceforth when we refer to Darboux transforms of a
constant mean curvature torus we will always assume that they are
defined on $T^2$, rather than merely on its universal cover.

As many authors have noted (see eg. \cite{Pohlmeyer:76, Uhlenbeck:89,
  Hitchin:90}), a harmonic map $ N \colon M\to S^2$ from a Riemann
surface $M$ into the $2$-sphere gives rise to a family of flat
connections $\nabla^{\zeta}$ on a trivial $\C^2$-bundle.  In the
quaternionic setting we view our trivial complex bundle as a
quaternionic line bundle $W$ over $M$ with a trivial connection
$\nabla$. Choosing a $\nabla$--parallel section
\newcommand{\pphi}{\phi} $\pphi$, the harmonic map $N$ gives rise to a
harmonic complex structure $J$ on $W$ by setting $J\pphi=\pphi N$.
Then the endomorphism $J$ is parallel with respect to the pull back
\[
\widetilde\nabla =\nabla +\frac 12 J^{-1}\nabla J
\]
under $N$ of the Levi-Civita connection of $S^2$.  Writing
$\omega=-\frac 12 J^{-1}\nabla J$, the family of flat connections is
given by
\begin {equation}
\nabla ^\zeta =\widetilde\nabla +\zeta\omega ^ {(1, 0)}+\zeta ^ {-1} \omega ^ {(0, 1)}, \quad \zeta\in\C_*\,,\label{eq:conns}
\end {equation}
where the type decomposition of $\omega$ is with respect to the
constant complex structure $I$ on $W$ which is right multiplication by
the quaternion $i$. However, in Lemma \ref {lem:family_unit_circle} we
encoded the harmonicity of a complex structure $J$ on a trivial
quaternionic line bundle as the requirement that the connections
(\ref{eq:complex_family})
\[
\nabla^\mu = \widetilde\nabla + (a + Jb) A + Q,\quad \mu = a + bI\in\C
_*
\]
have zero curvature. Recall here that $A = \frac 14(J \nabla J +
*\nabla J)$ and $Q= \frac 14(J\nabla J -*\nabla J)$ give the type
decomposition of $\omega$ with respect to the complex structure $J$ on
$W$.
\begin {lemma}
  The families of flat connections $\nabla ^\mu $ and $\nabla ^\zeta$
  defined above are gauge equivalent, with $\mu =\zeta ^ 2 $.
\end {lemma}
\begin {proof}
  Write $\zeta =u +v I $ with $ u=
  \frac{\zeta+\zeta\invers}{2}$, $v= \frac{\zeta\invers -
    \zeta}{2}I$ so that $ u ^ 2+ v ^ 2 =1$.  Then \eqref{eq:conns} shows that
\begin{eqnarray*}
  \nabla ^\zeta 
&=&\widetilde\nabla + (u+vJ) A + (u- vJ) Q.
\end{eqnarray*}
Define $\mu =\zeta ^ 2 $, and put $\lambda = a + b J$.  Furthermore,
denote by $\lambda ^\frac 12 = u + vJ $ the choice of square root
whose coefficients agree with those of $\zeta $.
Since $J$ is parallel with respect to $\widetilde\nabla $ we
see 
\[
\lambda ^ {\frac { 1} {4}}\cdot\nabla ^ \zeta =\lambda ^ {\frac { 1} {4}} (\widetilde\nabla +\lambda ^\frac {1} {2}A+\lambda ^{-\frac {1} {2}} Q)\lambda ^ {-\frac { 1} {4}}  = \widetilde\nabla +\lambda A +Q =\nabla ^\mu.
\]
Note that both choices of square root of $\lambda ^\frac 12 $ produce the same gauge. 
\end {proof}

We now return to those harmonic maps $N$ arising as Gauss maps of
constant mean curvature immersions $f\colon M\to \Im\H \simeq\R ^ 3
$. In this case $W=V/L$ where $L\subset V$ is the line subbundle of
the trivial $\H^2$-bundle $V$ given by the immersion $f$. As explained
earlier the point at infinity $e\H$ defines a splitting $V=L\oplus
\underline{e\H}$ which induces on $V/L=\underline{e\H}$ a trivial
connection $\nabla$ with parallel section $\pphi=e$. Specialising to
the case when $M=T^2$ is a torus for $\gamma\in \pi_1(T^2,p)$ let
\[
H ^\mu _\gamma (p)\in {\bf SL}(2,\C)
\]
be the holonomy of $\nabla ^\mu$ about $\gamma$ with base point $p\in T^2$.
For generic  $\mu\in\C_*$ there is a unique pair of lines $ {\mathcal {E}} _\mu (p) $ and $ \widetilde {{\mathcal {E}}} _\mu (p)
$ in $(V/L)_p $ which vary holomorphically in $\mu $ and are eigenlines
for $ H ^\mu _\gamma (p) $.   Since the fundamental group of the torus is abelian, the holonomy matrices for different generators $\gamma\in\pi _1 (T^2,p) $ commute, and so the eigenlines do not depend upon the choice of $\gamma $.

We will define the eigenline spectral curve of $ f $ by taking the minimally branched 2-sheeted cover of $\CP ^ 1$ on which these eigenlines, and their limits as $\mu\rightarrow 0,\infty $, are well-defined.
 To determine the branching over $\mu = 0,\infty $, we  investigate the limiting behaviour of the eigenlines and eigenvalues of the holonomy.
Denote by $ E (p)
 $ the $i$-eigenspace of $ J (p) $, then $
E(p) j $ is the $(-i)$-eigenspace.
\begin {theorem}
\label{thm:limiting}
\begin {enumerate}
\item
The (common) eigenlines of the holonomy $ H ^ \mu (p) $ have the following holomorphic limits:
\begin{eqnarray*}
\lim _ {\mu\rightarrow 0} \mathcal{ E} _ {\mu}(p) &=& \lim _ {\mu\rightarrow 0}\widetilde {\mathcal{ E}} _ {\mu}(p) ={ E }(p)j \\
 \lim _ {\mu\rightarrow \infty } \mathcal {E} _ {\mu} (p)&=& \lim _ {\mu\rightarrow \infty }{\widetilde {\mathcal {E}}} _ {\mu}(p)=  E(p).
\end{eqnarray*}
 \item These eigenlines each agree in the limit only to first order.
\item For $\gamma \in\pi_1 (T ^ 2) $, denote by $ h _\gamma (\zeta) $ the eigenvalues  of the holonomy
$ H ^\mu_\gamma $, where $\zeta ^ 2 =\mu $. There is a punctured neighbourhood $ U $ of
$\infty\in\C\P ^ 1 $ and $ w_{-1}\in\C_*$,  $ u_i:\pi_1 (T ^ 2)\rightarrow \C $ satisfying
\begin {align*}
\log h_\gamma (\zeta ) & =\pm (w_{- 1}\gamma\zeta + u_ 0 (\gamma) + u_ 1(\gamma) \zeta ^ {- 1}+\cdots)\,, \;  {\mbox{ for $ \zeta \in U $}}\\
\log h_\gamma (\zeta ) & = \mp (\bar w_{- 1}\gamma\zeta^ {- 1} + \bar u_ 0(\gamma) + \bar u_ 1(\gamma) \zeta+\cdots)\,, \; {\mbox{ for $\bar\zeta ^ {- 1}\in U $}}
\end {align*}
where  we interpret $\gamma $ as a complex number.
\end {enumerate}
\end {theorem}

\begin {proof}
We first give the proof of (i), which will include the statement and proof of Lemma~\ref {lemma:Psi} below.

  From \cite[Sec. 6.3]{klassiker} there are only finitely many $\mu$ for which
  the holonomy of $\nabla^\mu$ has just one eigenspace.  Thus we may
  let $U $ be a punctured neighbourhood of $\infty\in\CP^1$ consisting
  of $\mu $ for which the eigenlines of the holonomy $ H ^\mu $ are
  distinct.  The $(0,1)$ part of $\nabla^\mu$ with respect to the
  complex structure $I$ gives by \eqref{eq:nablaIdec} the complex
  holomorphic structure \[ (\nabla ^ \mu)^ {(0, 1)} =\nabla ^
  {(0, 1)} + (\mu ^ {- 1}-1)A ^ {(0, 1)}
\]
on $ V/L $ resulting in a complex rank two holomorphic vector bundle
over $ T^2\times(U \cup\{\infty\})$. The restriction $V/L_{(\cdot
  ,\mu)}$ is a bundle over $T^2$ with holomorphic structure
$(\nabla^\mu)^{(0,1)}$.

The argument of \cite[Prop.\hspace*{-1pt} 3.5] {Hitchin:90} may be applied here to show that $ U $ can be chosen so that for
each $\mu\in U $,
\[
\dim H ^ 0
\left(\End _0  (V/L_{(\cdot ,\mu)})\right)=1\,,
\]
and all holomorphic sections are in fact parallel with respect to
$\nabla^\mu$. Here $\End_0$ denotes the bundle of
trace free endomorphisms.
Thus we can take a family of non-trivial holomorphic sections
\begin {equation*}
\Psi (p,\mu) =\Psi_0(p) +\Psi_1(p)\mu^ {- 1} +\cdots 
\end {equation*}
of $\End_0(V/L_{(\cdot
    ,\mu)})$ for $\mu\in U \cup \{\infty\}$.
Since $\Psi$ is also $\nabla^\mu$--parallel, we may use it to investigate the limiting behaviour of the holonomy. We have
\begin {equation}
\left(\nabla ^ {(1,0)} + (\mu - 1)\ad_{A ^ {(1, 0)}} \right) (\Psi_0
 +\Psi_1\mu^ {- 1} +\cdots)=0\,.\label {eq:series}
\end {equation}
In particular
\begin {equation}
[ A ^ {(1, 0)},\Psi _0] =0 \label {eq:mu}
\end {equation}
and
\begin {equation}
 \nabla ^ {(1,0)}\Psi_0+ [A ^ {(1, 0)},\Psi_1] = 0\,.\label {eq:constantterm}
\end {equation}
In fact $ A ^ {(1, 0)} $ and $\Psi_0 $ are related by a {\it global} holomorphic differential, as we now show.
\begin {lemma} \label {lemma:Psi}Let $ dz $ be a global nontrivial holomorphic differential on the torus $ T ^ 2 $.
Then $\Psi (p,\mu) =\Psi_0(p) +\Psi_1(p)\mu^ {- 1} +\cdots $ may be chosen so that
\begin {equation*}
A ^ {(1, 0)} =\Psi_0dz.
\end {equation*}
\end {lemma}
\begin{proof}
We first show that $ A ^ {(1, 0)} $ is nowhere vanishing. Since $ -2 A
^ {(1, 0)} = (I + J)*A $ it suffices to show that $\im(*A)\nsubseteq
Ej $. But since $*A $ is right $\H $-linear and from \eqref
{eq:A=delta} is nowhere vanishing,   $\im(*A)\subseteq
Ej $  is impossible.
 Using \eqref {eq:mu},
 for any choice of parallel endomorphism $\Psi $, we have $\Psi_0 dz = b A ^ {(1, 0)}$ for some function $  b \colon T^2\to \C$. Then
\begin {equation}
 d ^\nabla (\Psi_0 dz)= \nabla \Psi_0\wedge dz = db\wedge A ^ {(1, 0)} + b d ^\nabla A ^ {(1, 0)}\label {eq:b}
\end {equation} and from
 \eqref{eq:harmonicity}
\[
0 = d ^\nabla*A = d ^\nabla (JA) =\nabla J\wedge A + J d ^\nabla A = - 2JA\wedge A + J d ^\nabla A,
\]
where the last equality used $\nabla J = 2 (*Q - * A) $ and $ Q\wedge A = 0 $. Thus
\begin {equation}
d ^\nabla A ^ {(1, 0)}= \frac{1}{2}d ^\nabla (A-I*A) = A\wedge A.\label {eq:1}
\end {equation}
Since $\Psi $ is holomorphic,
\[
\left(\nabla ^ {(0, 1)} + (\mu ^ {- 1}- 1)\ad_{A ^ {(1, 0)}} \right) (\Psi_0
 +\Psi_1\mu^ {- 1} +\cdots)=0,
\]
and equating constant terms  gives
\begin {equation}
 \nabla ^ {(0, 1)}\Psi_0 = [A ^ {(0, 1)},\Psi_0]. \label {eq:constantholomorphic}
\end {equation} Substituting this and \eqref {eq:1} into  \eqref {eq:b} gives $ db\wedge A ^ {(1, 0)} = 0 $ and hence $ \delbar b = 0 $ so we conclude that $ b$ is constant.
Scaling $\Psi $ by powers of $\mu $ if necessary, we may assume that $\Psi_0 $ is not the zero function and hence $ b\neq 0 $, proving the lemma.
\end {proof}
We henceforth choose $\Psi $ as in Lemma~\ref{lemma:Psi}.
Since $\Psi (p,\mu) $ is parallel, $\mathcal E_\mu (p)$ and
$\widetilde {\mathcal E}_\mu (p)$ are the eigenlines of $\Psi (p,\mu)
$.
Using $*A= JA = - A J$, we see
\[
\im A^ {(1, 0)} \subseteq  E\subseteq \ker A^ {(1, 0)}
\]
for the $+i$-eigenspace $E$ of $J$, and since $A^ {(1, 0)} $ is nowhere vanishing these are equalities.  Thus $A^ {(1, 0)}(p) $ is
nilpotent, with sole eigenvalue zero and just one eigenspace, namely $
E(p) $.  From Lemma ~\ref{lemma:Psi}, $ A ^ {(1, 0)}$ and $\Psi_0 $ have common eigenspaces, and we conclude that
\begin {equation}
\lim_{\mu\to\infty} {\mathcal {E}} _\mu(p) = \lim_{\mu\to\infty}\widetilde {\mathcal  E}_\mu (p)= E(p).\label {eq:E}
\end {equation}
From the reality condition \eqref{eq:reality} 
we see also that as $\mu\to 0 $, the limit of the eigenlines is $E(p)j$.
This concludes the proof of Theorem \ref{thm:limiting} (i).

We proceed now to the proof of (iii) of Theorem \ref{thm:limiting}, which will include the statement and proof of Lemma~\ref{Lemma:determinant}.
Let $ pr:\CP ^ 1\rightarrow\CP ^ 1 $ be the double cover $pr(\zeta)=\zeta^2 $. The eigenlines of $\Psi(p)$ define a line bundle $\mathcal E $ on $ T ^ 2\times (pr) ^ {- 1} (U) $, and writing $ \mathcal  E ^\zeta :=\left.\mathcal E\right|_{T ^ 2\times\{\zeta\}} $, the holomorphic structure on each $\mathcal  E ^\zeta $  is given by $ (\nabla ^ {\zeta ^ 2}) ^ {(0, 1)} $. From  \eqref {eq:E}, the bundle $\mathcal  E $ extends
holomorphically over $\zeta = \infty $.

For each $\zeta\in (pr) ^ {- 1} (U) $, the line bundle $\mathcal E ^\zeta $ carries the flat connection $\nabla ^ {\zeta ^ 2} $, so has degree 0 and hence is smoothly trivialisable. Since degree is constant in families, $\mathcal E ^\infty = E $ also has degree 0 and we may choose a holomorphic family of smooth trivialising sections
\begin{equation}
\label{eq:expansion}
 \varphi  (p,\zeta) = \varphi _0 (p) +\zeta ^{-1} \varphi _1 (p) +\zeta ^ {-2}\varphi_2 (p) +\cdots\mbox { for }\zeta\in (pr) ^ {- 1} (U)\cup\{\infty \}.
\end{equation}
Denote by $\Omega_\zeta $ the connection form of the restriction of $ \nabla ^ {\zeta ^ 2} $ to $\mathcal  E ^ {\zeta} $ with respect to this trivialisation, that is
\begin {equation}
(\nabla  + (\zeta ^ 2 -1) A ^ {(1, 0)} + (\zeta ^ {- 2} -1) A ^ {(0, 1)}) ( \varphi _0+\zeta ^ {- 1} \varphi _1 \cdots ) =\Omega_\zeta ( \varphi _0+\zeta^{-1} \varphi _1+ \cdots )\,,\label {eq:conn}
\end {equation}
and $\Omega_\zeta = \zeta ^ 2\omega _{-2}+ \zeta\omega _{-1} + \omega _0+ \zeta ^ {- 1}\omega_1+\cdots  $.
Let  $ k $ be a function on the universal cover $\C $ of $ T ^ 2 $ so that $ k (p)\varphi (p,\zeta) $ is parallel with respect to $\nabla ^ {\zeta ^ 2} $. Then
\[
dk  = -\Omega_\zeta k\,.
\]
Integrating gives
\[
\log h (\zeta) =\log\left (\frac {k (p +\gamma)} {k (p)}\right) = -\int_\gamma\Omega_\zeta.
\]
Thus
to prove the first equation in Theorem~\ref{thm:limiting}  (iii) it suffices to show that
for $\zeta $ near $\infty $, the connection form may be written as
\[
\Omega_\zeta = \zeta w_{- 1} dz + \omega_0 (p) + \zeta^ {- 1} \omega_1 (p) +\cdots
\mbox { for $ w _ { - 1} $ 
a non-zero constant.}
\]
The section $\varphi_0$ trivialises the bundle $\mathcal E ^\infty = E $ and hence is nowhere vanishing, so since $ A ^ {(1, 0)} $ has sole eigenvalue zero the $\zeta ^ 2 $ term of  \eqref {eq:conn} gives that $ \omega_{-2} = 0 $.
Using Lemma~\ref{lemma:Psi}, the $\zeta $ term of  \eqref {eq:conn} is
\begin {equation}\label {eq:connection}\Psi_0\varphi_1dz =\varphi_0\omega_{-1}.
\end {equation}
Using that $\Psi (p,\mu)  $ is trace-free,  
\begin {align}
\det\Psi (p,\mu) & =-\frac 12\tr\left (\Psi (p,\mu)^ 2\right)\nonumber\\
& =\det \Psi _ 0(p) -\tr(\Psi_0(p)\Psi_1(p)) \mu^ {- 1} + c _ 2\mu ^ {-2} +\cdots\nonumber\\
& =\mu^ {- 1} (\tr (\Psi_0 (p)\Psi_1 (p)) + c _1(p)\mu^ {- 1} +\cdots). \nonumber 
\end {align}
Writing $\zeta ^ 2 =\mu $, we see that
 the eigenvalues $\pm a (p,\zeta) $ of $\Psi (p,\mu) $ are of the form
\[
a (p,\zeta) = a_1 (p)\zeta^ {- 1} + a_2 (p)\zeta ^ {-2} +\cdots ,\,\text {where }
a_1 ^ 2 (p)  = -\tr (\Psi_0 (p)\Psi_1 (p)).
\]
The section $\varphi $ is an eigenvector of $\Psi $
 and the $\zeta^ {- 1} $ term of the eigenvector equation is
\begin {equation}
\Psi_0\varphi_1 = a_1\varphi_0.\label {eq:eigenvector}\end {equation} Combining this with \eqref {eq:connection} gives $ \omega_{-1} = a_1 dz $ so that
\[
\Omega_\zeta = \zeta a_{1} (p) dz + \omega_0 (p) + \zeta^ {- 1} \omega_1 (p) +\cdots,\,\text {where }a_1 ^ 2 (p)  = -\tr (\Psi_0 (p)\Psi_1 (p)).
\]
The first equation in Theorem~\ref{thm:limiting}  (iii) now follows from the following lemma.

\begin {lemma}
The trace $\tr (\Psi_0 \Psi_1) $ is a non-zero constant. 
\label{Lemma:determinant}
\end {lemma}
\begin {proof} We first show that this trace vanishes if and only if $\nabla  - A $ preserves  $E=\mathrm {ker} \,\Psi_0 $.
 From \eqref {eq:constantholomorphic} we know that $\ker \Psi_0 $ is preserved by $ (\nabla - A) ^ {(0, 1)} $. On the other hand  \eqref {eq:constantterm} and  Lemma~\ref {lemma:Psi} show that $ (\nabla - \ad_A) ^ {(1, 0)} \Psi_0 = [\Psi_1,\Psi_0 ] dz $. Thus $\nabla-A$ preserves $E$ if and only if $[\Psi_1,\Psi_0 ] $ is a multiple of $\Psi_0 $, which is equivalent to
$\tr (\Psi_0\Psi_1) =0$. Restricting
\[
(\nabla-\ad_A) J= 2QJ
\]
to $E$ and recalling that $Q$ interchanges $E$ and $Ej$, we have $\tr (\Psi_0\Psi_1) \neq 0$. Since $\Psi _0,\Psi_1 $ are holomorphic on $ T ^ 2 $, $\tr (\Psi_0\Psi_1) $ is constant.
\end{proof}

From the reality condition  \eqref {eq:reality} we have
\[
\log h (\bar\zeta ^ {- 1}) = -\log h (\zeta)\,,
\]
which completes the proof of Theorem~\ref{thm:limiting}   (iii). 
For (ii), we observe that
 since $ a_1 $ and $\varphi_0 $ are nowhere vanishing, from \eqref {eq:eigenvector} the same is true of $\varphi_1 $, so the eigenlines of the holonomy agree only to first-order at $\zeta =\infty $. This concludes our verification of Theorem~\ref {thm:limiting}.
\end {proof}

Let $ q $ be the unique (up to real scaling) quaternionic Hermitian
form on $ V/L$ that is parallel with respect to $\nabla $, and denote
by
\begin {equation*}
q= q_\C +j\det 
\end {equation*} its splitting into a complex valued hermitian form and a complex-linear non-degenerate 2-form
with respect to multiplication by the constant quaternion $ i$. We use $\det $ to measure the order to which eigenlines agree.
Define a polynomial
\[
P(\mu) =\prod (\mu -\mu _\alpha) ^{n_\alpha}
\]
by the condition that $\mu _ \alpha\in\C_* $ is a zero of $ P $ of
order $ n _\alpha $ if and only if $ {\mathcal {E}} _ {\mu _\alpha}
(p)$ and $ \widetilde {\mathcal {E}} _ {\mu _\alpha} (p)$ agree to
order $ n _\alpha $, as measured by the order of vanishing of the 2-form $\det$. 
By (\ref {eq:reality}) the polynomial $P$ is preserved by
$\mu\mapsto\bar\mu^ {- 1} $.  From Theorem \ref{thm:limiting} the
eigenlines of the holonomy have a unique limit at each of $\mu=0$ and
$\mu=\infty$ and  the two eigenlines agree there only to first-order. Thus we define the \emph{eigenline spectral curve} $
\Sigma _ e $ to be the curve 
\[
y^2 =\mu P (\mu)\,.
\]

By construction, for each $ p\in T ^ 2 $ the eigenlines of $ H ^\mu
_\gamma (p) $ 
define a line bundle $\mathcal E (p) $ on
$\Sigma _ e$ which by the above theorem extends holomorphically to
$P _\infty =y^ {- 1} (\infty) $ and $P _0 = y ^ {- 1} (0) $. We call
the line bundles $\mathcal E(p) $ {\em eigenline bundles}. The restriction of these bundles to an open set without singular points is holomorphic, and since $\Sigma_e $ is smooth in a neighbourhood of  $ P_\infty $ and $ P_0 $, we may define the eigenline bundles over these points by holomorphic extension. The
symmetry of $P$ ensures that $\Sigma _ e $ possesses the fixed point
free real structure $\rho$ corresponding to the action of the
quaternion $j$ on $V/L$, that is, $\rho^* \mathcal E(p)=\mathcal
E(p)j$.  We can similarly define a polynomial $ Q (\zeta) $ using the
holonomy of the family $\nabla ^\zeta $.  Taking $\mu =\zeta ^2$,
since the connections $\nabla ^\mu $ and $\nabla ^\zeta $ are gauge
equivalent, the two eigenlines of the holonomy of $\nabla ^\zeta $
agree to the same order as the eigenlines of the holonomy of $\nabla
^\mu $.  Thus
\[
Q (\zeta) = P (\mu).
\]
Define a hyperelliptic curve $ Y$ by
\[
\nu^ 2 =Q(\zeta);
\]
this is the spectral curve used in \cite {Hitchin:90}.  We have
shown above that $ Q $ is preserved under $\zeta\mapsto -\zeta $ and we denote by
 $\tau(\zeta,\nu)=(-\zeta,-\nu) $  the corresponding fixed point free holomorphic involution of $Y$.
 We then obtain
\begin {corollary}
\label {cor:genus}
The \eigenline spectral curve is the quotient
\[
 \Sigma_e \cong {Y}/{\tau}
\]
of the curve $Y$  by a holomorphic involution without fixed points.

\end {corollary}
The key property of these curves is that they have finite genus, and
so we are in the realm of algebraic geometry.  This was proven for $Y$
in \cite {Hitchin:90} and for $\Sigma _ e $ in \cite[Sec.\hspace{-2pt} 6.3]
{klassiker}.  We note that working with a quaternionic line bundle
rather than a complex rank two vector bundle makes the finite genus
result easier to prove for $\Sigma _ e $ than to prove
directly for $ Y $, so this corollary yields a simplification of the
proof in \cite {Hitchin:90}.

Denote by $\Sigma _ e ^ \circ$ the open \eigenline spectral
curve $\Sigma _ e\setminus \{P_0,P_\infty\} $.
\begin {theorem}
\label{thm:parametrized mu Darboux}
For a constant mean curvature immersion $ f\colon T ^ 2\rightarrow\R ^
3 $ of a 2--torus the space of $\mu $-Darboux transforms $\hat{f}\colon T^2\to \R^4$ of $ f $ is
given by the quotient of its open \eigenline spectral curve by the
(fixed point free) real structure $\rho $, together with finitely many
complex projective lines
\[
\{ \text{ $\mu$-Darboux transforms $\hat{f}\colon T^2\to \R^4$, } \mu\in\C_*\} =\Sigma^{\circ}_e / \rho \cup \CP ^ 1\cup\ldots\cup \CP^1,
\]
where the projective lines are distinct and each intersects
$\Sigma^{\circ}_e / \rho $ in one or two points. The pair
$(P_0,P_{\infty})$ corresponds to the original immersion
$f$.
\end {theorem}

\begin {proof}
  \begin {comment} {\bf $\mu$-parallel sections in $E$ are zero, hence
      never get $f$ on open eigenline curve as DT;}
\end {comment}
For each $ x  \in\Sigma _ e ^ \circ $, choose a $\nabla^{\mu(x)}
$-parallel section $\varphi ^ x $ satisfying $\varphi ^ x (p)\C = {\mathcal {E}} _x
(p) $ for $p\in T^2$, and let $\widehat\varphi ^ x $ be the prolongation
\eqref{eq:prolongation} of $\varphi ^ x $.  The map $\hat f ^ x
=\widehat\varphi ^ x\H \colon T^2 \to S^4 $ is by definition a $\mu
$-Darboux transform of $f$. If the holonomy $H ^\mu_\gamma $ has
distinct eigenspaces, then $x $ clearly uniquely determines $\hat f ^
x $. From
 \cite{klassiker} we  know that there are only finitely many $ x  \in\Sigma _ e ^ \circ $
such that the holonomy has a two-dimensional eigenspace and hence
 we may extend the map $x\mapsto \hat f ^ x $ to such points.  By
(\ref {eq:reality}), the section $\varphi ^ x j$ is parallel for
$\nabla^{\mu(\rho(x))}$, so the points $x$ and $\rho(x)$ give rise to
the same $\mu$--Darboux transform and we obtain a well--defined map
\begin{eqnarray*}
  \Sigma^{\circ}_e / \rho & \rightarrow & \{\text{$\mu$-Darboux transforms $\hat{f}\colon T^2\to \R^4$ of $f$}\}\\
x &\mapsto &\hat f ^ x =\hat\varphi ^ x\H\,.
\end{eqnarray*}
It is proven in \cite {Hitchin:90} that $\rho $ acts without fixed
points. In Theorem~\ref{theorem:extends to f} we prove that the
points $ P _ 0 $ and $ P _\infty $ correspond to the original
immersion $ f $.

Suppose that $\hat f ^ {x_1} =\hat f ^ {x_2 } $, i.e. $\hat\varphi ^
  {\mu_2} =\hat\varphi ^ {\mu_1} g $ for $g:\R^2\rightarrow\H_*
  $. Then
\[
\pi d\hat\varphi ^ {\mu_2} = (\pi d\hat\varphi ^ {\mu_1})g +\pi\hat\varphi ^ {\mu_1} dg
\]
so from \eqref{eq:prolongation} we see that $g = v + w j $
 is constant. Thus
\begin{eqnarray*}
0 & = &\nabla ^ {\mu_2} (\varphi_1 g) =(\nabla ^ {\mu_2} \varphi_1) v + (\nabla ^ {\bar\mu_2^ {- 1}}\varphi_1) w j\\
& = & (\mu_2 -\mu_1) A ^ {(1, 0)}\varphi_1v + (\bar\mu_2^ {- 1} -\mu_1) A ^ {(1, 0)}\varphi_1wj + (\mu_2 ^ {- 1} -\mu_1^ {- 1}) A ^ {(0, 1)}\varphi_1v \\
& & + (\bar\mu_2 -\mu_1^ {- 1}) A ^ {(0, 1)}\varphi_1wj,
\end{eqnarray*}
where the first and last terms take values in $E = \im A ^ {(1,
  0)} $, and the remaining terms are valued in $Ej = \im A ^
{(0, 1)} $. Since the $(1, 0) $ and $(0, 1) $ parts each separately
vanish, each of the four terms above is zero and so $x_2 $ is either
$x_1 $ or $\rho (x_1) $.

Suppose that $\mu\in S^1$ is such that the holonomy matrix $ H ^{\mu}
 _\gamma (p) = H ^{\rho (\mu)} _\gamma (p) $ has a two-dimensional
 eigenspace.  The connection $\nabla^{\mu}$ has ${\bf SU}(2)$ holonomy and
 the limiting lines $ \mathcal E_{\mu} =\lim_{\eta\to\mu} \mathcal
 E_\eta $ and $\widetilde { \mathcal E} _{\mu}
 =\lim_{\eta\to\mu}\widetilde { \mathcal E} _\eta $ satisfy
 $\widetilde { \mathcal E} _{\mu}= \mathcal E_{\mu} j$. In particular
 $\mathcal E _\mu $ and $\mathcal E _\mu j $ span the two-dimensional
 eigenspace and each point in this eigenspace yields the same $\mu
 $-Darboux transform, corresponding to the point in $\Sigma _ e/\rho $
 with this $\mu $-value.

 We turn our attention then to the finitely many $\mu\in\C_*\setminus S^1$ for which the
 holonomy matrices $H ^{\mu} _\gamma (p)$ and $ H ^{\rho (\mu)}
 _\gamma (p) $ each have two-dimensional eigenspaces.
  In this case, writing $\mathcal E _\mu $ again for the limiting
 eigenline, $\mathcal E _\mu j $ does not belong to the
 two-dimensional eigenspace $W $ of $ H ^\mu $.  Hence the
 $\mu$--Darboux transforms associated to the pair $ (\mu,\rho (\mu)) $
 are parameterised by $\P (W)\simeq\CP ^ 1 $.  If $\Sigma_e^{\circ}$ is
 branched over $\mu$ there is exactly one $\mu$--Darboux transform
 given by the pair $ (x,\rho (x))\in\Sigma_e^\circ$ with $\mu(x)=\mu$, and
 the $\CP^1$ intersects $\Sigma_e^\circ/\rho$ in a single point.  When
 $\Sigma_e^\circ$ is unbranched over $\mu\in\C_*\setminus S^1 $, we have two
 limiting eigenlines ${\mathcal {E}} _{\mu}$ and $\tilde {\mathcal
   {E}} _{\mu} \not= {\mathcal {E}} _{\mu} j$ and the $\CP^1$
 intersects $\Sigma_e^\circ/\rho$ in two points.
 \end{proof}

\section {The multiplier spectral curve}

We introduce a one-dimensional analytic variety, called the {\it
  multiplier spectral curve} \cite{conformal_tori}, which has the
advantage that it may be defined for any conformal immersion of a
2-torus into $ S ^ 4 $ with degree zero normal bundle. In general this
variety may have infinite geometric genus, but we show that in the
case of a constant mean curvature torus in $\R^3$ its
geometric genus is finite. Indeed, its normalisation completes to a compact
Riemann surface biholomorphic to the normalisation of the \eigenline
curve.  The
multiplier curve and the \eigenline curve are not in general isomorphic since the multiplier curve
is always singular (Corollary~\ref{cor:singular}).

As $\pi _ 1 (T ^ 2) $ is abelian the holonomy of a section of
$\widetilde {V/L} $ (the pullback of $ V/L $ to the universal cover  $\C $ of $ T ^ 2 $)
lies in an abelian subgroup of $\H_*$. We assume that we have
conjugated so that this subgroup is equal to $\C _*$.
\begin {definition}
  The {\it multiplier spectral curve} $\Sigma_m$ of a conformal
  immersion $ f: T ^ 2\rightarrow S ^ 4\simeq \HP^1 $ is the set of
  holonomies realised by holomorphic sections of $\widetilde {V/L}
  $. Denote by $ H^0_h(\widetilde{V/L}) $ the space of holomorphic
  sections $\varphi$ with multiplier $h\in \Hom(\pi _ 1 (T ^ 2),\C_*)$, that is
  $\gamma^*\varphi = \varphi h_\gamma$ for all $\gamma\in\pi_1(T^2)$.
Then
\begin{eqnarray*}
\Sigma_m &=& \{ h\in \Hom(\pi _ 1 (T ^ 2),\C_*) | \text{ there exists } 0\not\equiv\varphi
\in H^0_h(\widetilde{V/L}) \}\,.
\end{eqnarray*}
\end {definition}
Let $\Harm (T ^ 2,\C) $ be the space of harmonic 1-forms on $ T ^ 2 $, then there is a natural isomorphism
\begin{eqnarray*}
\Harm(T ^ 2,\C) / \Gamma^* &\rightarrow & \Hom(\pi _ 1 (T ^ 2),\C_*)\\
\omega &\mapsto &\exp{\int\omega}
\end{eqnarray*}
where $\Gamma ^*$ denotes the harmonic 1-forms with periods in $2\pi i \Z$. The lift
$\log\Sigma _ m $ of the multiplier spectral curve to $\Harm(T ^ 2,\C)
$ is the space on which a holomorphic family of
elliptic operators has nontrivial kernel \cite {conformal_tori, ana}.  One
concludes that $\log\Sigma _ m $ and hence $\Sigma _ m =\log\Sigma
_ m /\Gamma ^* $ are one-dimensional analytic varieties, justifying
our terminology.  Notice that these curves possess a real structure $\sigma$
given by complex conjugation of the holonomy.

The multiplier curve allows us to give a geometric description of the space of  Darboux transforms.
\begin {theorem}\label {thm:spectral Darboux} For an immersed torus $f: T^2\to \R^3$ of constant
  mean curvature the set of all  Darboux transforms of $f$ is
  parameterised by the quotient of the multiplier spectral curve under the
  real structure $\sigma$ together with at most countably many complex and quaternionic projective spaces:
\[
\{ \text{Darboux transforms}\} =\Sigma_m/\sigma \cup \bigcup _ {m = 1} ^ {\infty} \CP^{k_m}
\cup\bigcup _ {n = 1} ^ {\infty }\HP^{l_n}.
\]
\end{theorem}
\begin {proof} A section of $\widetilde {V/L} $ is holomorphic when it
  lies in the kernel of the quaternionic holomorphic structure $D$ (introduced in \eqref{eq:holomorphic structure}) with multiplier $ h $. For all but a
  discrete set of $ h \in\Sigma_m$, the space $H^0_h(\widetilde{V/L})$
  is only complex one-dimensional \cite[Theorem
  3.3]{conformal_tori}.  The Darboux transform given by a holomorphic
  section $\varphi $ of $\widetilde {V/L} $ is unchanged by
  quaternionic scaling of this section, and if $\varphi $ has complex
  multiplier $ h $, then $\varphi j $ has multiplier $\bar h =\sigma (h)
  $. Thus away from a discrete set, to each pair $
  (h,\sigma(h))\in\Sigma_m/\sigma $ there corresponds a unique Darboux
  transform.

  If $H^0_h(\widetilde {V/L}) $ has complex dimension $k +1 $, then
  the same is true also for $\sigma (h) =\bar h $ since multiplying by $
  j $ interchanges the two spaces of sections.  Thus if $ h\not \in\R
  $, the space of Darboux transforms with multiplier $h $ or $\sigma(h) $
  is parameterised by $\CP ^k
  $. 
  If $ h\in\R
  $ the space $H^0_h(\widetilde {V/L})$ is a quaternionic vector
  space and thus the set of corresponding Darboux transforms is
  parameterised by $\HP ^ {l} $ with $l=\frac{k-1}2$.
\end {proof}

The eigenvalues of $ H ^\mu _\gamma (p) $ give a well-defined
holomorphic function $ h _\gamma (x) $ on $\Sigma ^\circ _{e}$, and we
write
\[
\begin {array} {rccl}
 h  \colon &\Sigma ^\circ _{e}&\rightarrow &\Sigma_m\\
& x &\mapsto & (h \colon \gamma\mapsto h_\gamma (x))\,.\end {array}
\]

Let $\widetilde\Sigma_e $ denote the normalisation of the \eigenline
spectral curve $\Sigma_e $ and note as before that any
singularities of $\Sigma_e $ are contained in $\Sigma ^\circ _e $. We
write $\widetilde\Sigma_m $ and $\log\widetilde\Sigma _ m $ for the
(analytic) normalisations of the multiplier spectral curve $\Sigma _ m
$ and of $\log\Sigma _ m $,  and observe that
$
\log\widetilde\Sigma _ m/\Gamma ^* = \widetilde\Sigma _ m$.
 Let
 \[
\tilde h
:\widetilde\Sigma ^\circ _e \rightarrow\widetilde\Sigma_m
\]
be the lifting of $ h $ to the normalisations.

\begin {theorem}\label{thm:isomorphism}
  The multiplier curve of a constant mean curvature torus is
  connected, and its normalisation can be completed to a compact
  Riemann surface biholomorphic to the normalisation of the \eigenline
  spectral curve. This biholomorphism is given by (an extension of)
  the map $\tilde h $ defined above.
\end {theorem}

\begin {proof} 
  We show first that $\tilde h $ is injective.  It suffices to prove
  this away from a discrete set, as $\tilde h $ is a holomorphic map
  between two Riemann surfaces.  Write $\pi _e:\widetilde\Sigma _
  e\rightarrow\Sigma _ e $ for the normalisation map.  We show that
  $\tilde h $ is injective on
\[
U =\{x\in\widetilde\Sigma^\circ_{e} \colon \dim H ^ 0 _{ h   (\pi_e(x))} (\widetilde{V/L}) = 1\} \setminus\pi _e^ {- 1} (S_{e}\cup h   ^ {-1} (S_m)),
\]
where $S_{e} $, $ S_m $ are the singular points of the two
spectral curves. Since we are omitting these points, we do not need to
distinguish between $ h $ and $\tilde h $. As mentioned before, the set
$\{h\in\Sigma_m\colon \dim H^0_{h} (\widetilde{V/L}) > 1\} $ is
discrete \cite {conformal_tori}.  Its pre-image under the
holomorphic map $\tilde h $ is thus also discrete, and so $ U $ is the
complement of a discrete set.

Each $ x \in U $ determines a unique $\nabla^{\mu(x)} $--parallel section
$\varphi ^x\in H^0_{ h (x)} (V/L) $ up to complex scaling. If
\begin{equation}
\label{eq:equal_h}
h   ( x) = h   (x ^ {\#})
\end{equation}
then $\varphi ^x =\varphi^ {x ^\#} \lambda $ with $\lambda\in\C_*$ since
$\dim H^0_{h(x)}(\widetilde{V/L}) =1$. In particular $\varphi^x$ is
parallel with respect to both $\nabla ^{\mu(x)} $ and $\nabla ^ {\mu(x^ \#)}
$. Therefore \eqref{eq:nablaIdec} and the fact that $\ker A^{(1,0)}\cap \ker A^{(0,1)}=\{0\}$ imply that
$\mu (x^ \#) =\mu(x) $. If $x$ and $x^ \#$ are exchanged by the hyperelliptic involution then by \eqref{eq:equal_h} we see that $h   ( x) = h   (x ^ {\#})=\pm 1$.  Since $\dim H^0_{h(x)}(\widetilde{V/L}) =1$ we conclude that $x=x^\#$ is a branch point of $\mu$ and so $\tilde{h}$ is injective.

We now show that the map $\tilde h$ holomorphically extends to $P_0 =
 \mu\invers(0)$ after completing one end of the multiplier spectral curve
 by a single point.  It suffices to extend $\tilde h_\gamma$ for an
 appropriate $\gamma\in\Gamma$. In this proof we will not notationally
 distinguish between $h_\gamma$ and $h$.  The eigenline spectral curve
 $\Sigma_e$ is branched at $P_0$ hence $\zeta$ with $\zeta^2 = \mu$ is
 a local coordinate near $P_0$.  Theorem~\ref {thm:limiting} shows that $\log h(\zeta)$ is a chart around the end of the
multiplier spectral curve which is completed by adding the point
$ 0 $ in this chart. An analogous argument can be used at
$P_\infty$ and thus the normalised multiplier spectral curve becomes a compact
Riemann surface $\overline{\widetilde\Sigma}_m$ by adding these two
points. Moreover, $\tilde h\colon \widetilde\Sigma_e \to
\overline{\widetilde\Sigma}_m$ extends to an injective holomorphic
map.  For a general conformal branched immersion $ T ^ 2\rightarrow S
^ 4 $, the multiplier spectral curve $\Sigma _ m $ consists either of
two components each of which has one end, or of a single component
with two ends \cite {conformal_tori}. Since both ends of $\widetilde
\Sigma_m$ are contained in the image of $\tilde h$ the multiplier
spectral curve of a constant mean curvature torus is connected and
$\tilde h $ is a biholomorphism.
\end {proof}

The space $H^0_h(\widetilde{V/L})$ of holomorphic sections with
multiplier $h$ is generically complex one-dimensional. It defines a
line bundle $\mathcal L\to \widetilde\Sigma_m$, the {\em kernel bundle}, on the
normalisation of the multiplier spectral curve \cite[Theorem
3.6]{conformal_tori} with fibres $\mathcal{L}_{\tilde h} =\varphi\C$ where $\varphi \in H^0_h(\widetilde{V/L})$
and $h$ denotes the image of $\tilde{h}\in\widetilde\Sigma_m$ under the normalisation map $\pi_m\colon \widetilde\Sigma_m\to \Sigma_m$.
By Theorem~\ref{thm:isomorphism} each point in $\widetilde\Sigma_m$ is $\tilde{h}(x)$ for a $x\in\widetilde\Sigma_e$, where $\tilde{h}$ denotes the biholomorphism of the Theorem. Hence $\varphi$ is parallel with respect to $\nabla^{\mu(x)}$ and therefore has no zeros.  This enables us to define for each $p\in T^2$ a holomorphic line  bundle ${\mathcal L}(p)$ where  $\mathcal L_{\tilde h}(p)=\varphi(p) \C$.

\begin{corollary}
Let $f\colon T^2\to \R^3$ be a constant mean curvature torus.  For each $p\in T^2$ the push forward of the eigenline bundle ${\mathcal E}(p)$ under the biholomorphism $\tilde{h}$ is the evaluation ${\mathcal L}(p)$ of the kernel bundle
  \[
\tilde h_* {\mathcal E}(p)
={\mathcal L}(p)\,.
\]
Consequently, every Darboux transform $\hat f $ of
 $ f $ given by a holomorphic
  section in $\mathcal L_{\tilde h}$ with $\tilde h\in\widetilde \Sigma_m$  is a $\mu $-Darboux transform of $f$.

\end{corollary}

As we have seen the normalisations of the \eigenline and the
multiplier spectral curve are biholomorphic.   However the
two spectral curves are not in general isomorphic.
\begin {corollary}
\label{cor:singular}
  The multiplier spectral curve of a constant mean curvature torus is
  always singular.  The eigenline spectral curve is a partial
  desingular\-isation of the multiplier spectral curve.
\end {corollary}

\begin {proof} The projection of $\Sigma _ e $ to the $ \mu $-plane
  has at least one pair $ x,\rho (x) $ of ramification points over
  some $\mu,\bar\mu^ {- 1}\in \C_*\setminus S ^ 1 $.  At these points,
  for each generator $\gamma\in\pi _1 (T^2,p)$ the eigenvalues of the holonomy matrix $ H
  ^\mu _\gamma (p)\in S L (2,\C) $ are either both $1$ or $-1$. Thus
  the representation $ h(x)\colon\pi _1 (T^2) \rightarrow\C_*$ is in particular real and so $h(\rho
    (x)) =\overline{ h(x)} = h(x)$, showing that the map $
  h\colon\Sigma ^\circ _ e\rightarrow\Sigma _ m $ is not
  one-to-one.
  Since $\Sigma _e $ is smooth at $P_0$ and $ P _\infty $ the normalisation maps give an extension
  $\pi _m\circ\pi _ e^ {- 1}$  of $h $ to all of $\Sigma _
  e $ .  We then have that the normalisation map for $\Sigma _ m$
  factors as $\pi _ m =h\circ\pi _e $ and that it is not one-to-one,
  proving our claims.  We note that if we replace the original torus
  by a four-fold cover then all ramification points other than $P_0, P
  _\infty $ map to the constant multiplier $1 $.
  \end {proof}

\section {Geometric Picture}
The spectral curves encode geometric information about the original
constant mean curvature torus $f:T^2\rightarrow\R ^ 3 $. In this
section we show that $f$ is the limit of $\mu$--Darboux transforms as
$\mu$ tends to  $0$ or $\infty$.

The eigenline bundle $\mathcal E(p) \to  \Sigma_e$ based at
$p\in T^2$ gives rise to a complex line bundle $\mathcal E \to T^2
\times \Sigma_e^\circ$ via parallel transport of the fiber
$\mathcal E(p)_x$ over $x\in \Sigma_e^\circ$ by the connection
$\nabla^{\mu(x)}$. Since $\nabla^{\mu(x)}$ has a simple pole at $P_0$
and at $P_\infty$, parallel transport is only defined on
$\Sigma_e^\circ =  \Sigma_e\setminus\{P_0,
P_\infty\}$.

Then $\mathcal E_x =\varphi^x\C$ where $\varphi^x$ is a
$\nabla^{\mu(x)}$--parallel section with multiplier $h(x)$. The
prolongation $\hat \varphi^x$ of $\varphi^x$ (see Lemma
\ref{lem:prolongation}) defines a complex line subbundle
\[
\hat{\mathcal E} \to T^2\times \Sigma_e^\circ
\]
of the
trivial $\H^2$ bundle $V$ which is holomorphic over
$\Sigma_e^\circ$. The Darboux transform for
$x\in\Sigma_e^\circ$ is the map $\hat f^x = \hat{\mathcal
  E}_x \H \colon T^2 \to S^4$.
\begin {theorem}
  \label{theorem:extends to f}
  Let $f\colon T^2\to\R^3$ be a constant mean curvature torus with corresponding quaternionic line subbundle $L\subset V$
  and $\delta\colon L\to K V/L$ be the derivative of $f$ as defined in \eqref{eq:delta}.
  For every $p\in T^2$ the   line
  bundle $\widehat{\mathcal E }(p)\to \Sigma_e^\circ$
  extends holomorphically across $P_0$ and $P_\infty$ with
\[
\lim_{x\to P_0 } \widehat{\mathcal E }_x = \delta^{-1}(Ej) \quad \text{ and } \quad \lim_{x\to P_\infty}
\widehat{\mathcal E}_x = \delta^{-1}(E)\,.
\]
Hence, when $x$ tends to $P_0$ or $P_\infty$, the Darboux transforms $\hat{f}^x$  limit to the original constant mean curvature torus $f$.
\end {theorem}
\begin {proof}
As $f$ is an immersion $\delta$ has no zeros, and
  from \eqref{eq:Bmuc} and \eqref{eq:nablaIdec} we see that
$\widehat{\varphi}^x = \varphi^x -
\delta\invers(\alpha^\mu\varphi^x)\,, $ where
\[
\alpha^\mu = \nabla^\mu -\nabla = \mu A^{(1,0)} + \mu\invers
A^{(0,1)} - A\,.
\]
From \eqref{eq:expansion} we have
\[
\varphi^x = \varphi_0 + \varphi_1 \zeta\invers + \ldots
\]
with $\varphi ^ x\C = \mathcal E _ x $. From Theorem~\ref {thm:limiting}  we know that
$\varphi_0$ is a nowhere vanishing section of $E=\ker A^ {(1, 0)} $.
Then
\begin {align*}
\widehat\varphi ^x & = (\varphi _0+\varphi _1\zeta ^ {- 1} +\cdots - \delta^ {- 1} (\zeta ^ 2 A ^ {(1, 0)} +\zeta ^ {- 2}A ^ {(0, 1)} -A) (\varphi _0+\varphi _1\zeta ^ {- 1} +\cdots))\\
& =- \zeta \delta^ {- 1} ( A ^ {(1, 0)}\varphi _ 1) +\mbox {
lower order terms}
\end {align*}
so using Lemma~\ref{lemma:Psi}, Lemma~\ref{Lemma:determinant} and \eqref {eq:eigenvector} we have $ A^{(1,0)}\varphi_1 \neq 0 $. Since $\im A^{(1,0)}=E$  we conclude that $\lim_{x\to P_\infty} \widehat{\mathcal E }_x = \delta^{-1}(E)$. Furthermore, the limit of $\hat f^x = \hat{\mathcal
  E}_x \H$ as $x$ tends to $P_0$ or $P_{\infty}$ is $\delta^{-1}(E\oplus Ej)=L$.
\end {proof}

Historically, points on the eigenline curve could be interpreted
differential geometrically only for unitary $\mu$, where they correspond to the associated
family of constant mean curvature surfaces. From our results we obtain
an interpretation of all points on the spectral curve of a constant
mean curvature torus as its $\mu$--Darboux transforms. We combine
those observations in the following

\begin{theorem} Let $f: T^2\to\R^3$ be a constant mean curvature
  immersion and $\pi\colon \C\P^3\to\H\P^1$ the twistor projection which sends a  complex line in $\C^4$ to the corresponding quaternionic line in $\H^2$.  The maps
\[
\xymatrix {&\CP ^ 3\ar [d] ^\pi\\
T ^ 2\times\Sigma _e\ar [r]_{\hat{\mathcal E}\H} \ar [ru]^{\widehat {\mathcal E}}
&\HP ^ 1}
\]
defined by the prolongation of the eigenline bundles together with the
twistor projection satisfy:
\begin {enumerate}
\item For each $ x\in\Sigma ^\circ _ e $, $ \hat f^x =
  \pi\widehat{ \mathcal E }(\cdot, x)$ is a $\mu $-Darboux transform
  of the original constant mean curvature immersion $f$. In fact,
  $\hat f^x$ is also a constant mean curvature torus in Euclidean
  3--space (Theorem~\ref{thm:DT is CMC}).
\item The original constant mean curvature torus $f=\pi\widehat
  {\mathcal E} (\cdot, P _\infty)=\pi\widehat {\mathcal E} (\cdot, P
  _0)$ is the limit of $\mu$--Darboux transforms for $\mu\to 0,
  \infty$.
\item The Gauss map of $ f $ is given by $\mathcal E (\cdot ,P _\infty
  )=E$ as the $+i$ eigenspace of $J$.
\item For $p\in T^2$ the eigenline curve is algebraically mapped
  into $\CP^3$ by $\hat{\mathcal E}(p, \cdot)$. Thus we obtain a
  smooth $T^2$--family of algebraic curves $ \Sigma_e\to \CP^3$.
\end {enumerate}

\end {theorem}

By pulling back the eigenline bundles and their prolongations, the analogous result holds on the normalisation $\widetilde\Sigma_e\cong \widetilde\Sigma_m $. This normalised version
holds more generally for conformally
immersed tori into $S^4$ of finite spectral genus
\cite{conformal_tori}, for which there is a multiplier  curve but no eigenline curve. The proof of this result is much more involved
and requires asymptotic analysis of Dirac--type operators \cite{ana}.
For constant mean curvature tori our proof can be seen as a geometric interpretation of the eigenline spectral curve
\cite{Hitchin:90}. In fact, the harmonic Gauss map of the constant
mean curvature torus is described by the $T^2$--family of algebraic
functions $ \mathcal E\colon T^2 \times\Sigma_e \to \CP^1$
which, interpreted as a flow of line bundles, is linear in
the Jacobian of ${\Sigma}_e$.  From this one could show that
the prolongation bundle $\hat{\mathcal E}$ also gives a linear
$T^2$--flow in the Jacobian. In the generic case when $\Sigma_e $ is smooth, such flows can be parametrised by theta
functions of ${\Sigma_e}$, providing explicit conformal
parametrisations of the constant mean curvature torus and all its spectral
$\mu $--Darboux transforms \cite{Bobenko:91}.

\begin{appendix}
\section{Darboux Transforms of   the Standard Cylinder}

\newcommand{\eiy}{e^{iy}}  

We now illustrate some of our results by explicitly computing $\mu$--Darboux
transforms of the standard cylinder
\begin{equation}
\label{eq:cylinder}
f(x,y) = \frac{1}{2}(-ix  + j e^{iy})
\end{equation}
in $\R^3$. The derivative of $f$ is  
\[
df =  \frac{1}{2}(-i dx + j i e^{iy}dy)
\]
and the  Gauss map of $f$ is given by
\[
N(x,y) = - j \eiy\,.
\]
To compute $\mu$-Darboux transforms of $f$, we need to find parallel sections $\varphi\in\Gamma(\widetilde{V/L})$ of the
flat connection $\nabla^\mu = \nabla+ *A(J(a-1) + b)$ on $V/L$, where $a$ and $b$ are defined in terms of $\mu$ (Remark~\ref{rem:lambda}). As in section~\ref{sec:cmc}
we identify $V/L$ with $\Li$ via the splitting $V= L\oplus \Li$ where
$e =\begin{pmatrix} 1\\0
\end{pmatrix}
$. Thus, writing $\varphi =e\alpha$ we seek $\alpha: \tilde M =
\R^2\to\H$ with \eqref{eq:A=delta} 
\[
d\alpha = -\frac 12df(N\alpha(a-1) + \alpha b)\,.
\]
Using the decomposition $\alpha=\alpha_0+j\alpha_1$ where $\alpha_0,\alpha_1: \R^2
\to\C=\Span\{1,i\}$, we rewrite the previous equation as
\begin{eqnarray*}
\begin{pmatrix} \alpha_0\\ \alpha_1
\end{pmatrix}_x &=& 
\frac{i}{4}\begin{pmatrix} b &   e^{-iy} (a-1)\\
\eiy (a-1) &-b
\end{pmatrix}
\begin{pmatrix}\alpha_0\\ \alpha_1
\end{pmatrix}
\\[.4cm]
\begin{pmatrix} \alpha_0\\ \alpha_1
\end{pmatrix}_y &=& \frac{i}{4}
\begin{pmatrix} a-1 &  -e^{-iy} b\\
 -\eiy b&\displaystyle 1-a
\end{pmatrix}
\begin{pmatrix}\alpha_0\\ \alpha_1
\end{pmatrix}\,.
\end{eqnarray*}
The first differential equation has constant coefficients in $x$
and therefore we can solve this system explicitly. Choosing a square
root $c\in\C$ of $a-1$, the general
solution of this system on $\R^2$ is given by
\newcommand{\epivh}{e^{\frac{iy}{2}}}
\newcommand{\epzh}{e^{w}}
\newcommand{\emivh}{e^{\frac{-iy}{2}}}
\newcommand{\emzh}{e^{-w}}
\[
\alpha = \begin{pmatrix} \emivh\\
p_+ \epivh
\end{pmatrix}m_+ \epzh + \begin{pmatrix} \emivh\\
p_-\epivh
\end{pmatrix} m_-\emzh
\]
where $m_\pm\in\C$ and
\begin{equation}
\label{eq:wp}
w = \frac{\sqrt{2}}{4c}((a-1)x-by), \quad
p_\pm = - \frac{b\pm \sqrt{2}ic}{a-1}\,.
\end{equation}
Substituting the above formulae for $\alpha$
into  \eqref{eq:Tinvers} yields all $\mu$--Darboux transforms $\hat f = f+ T$ of
the standard cylinder. In general, these $\hat{f}$ do not satisfy any periodicity conditions despite the fact that the above differential equations have periodic coefficients.

\begin{figure}[h]
\begin{center}
\includegraphics[width=0.45\linewidth]{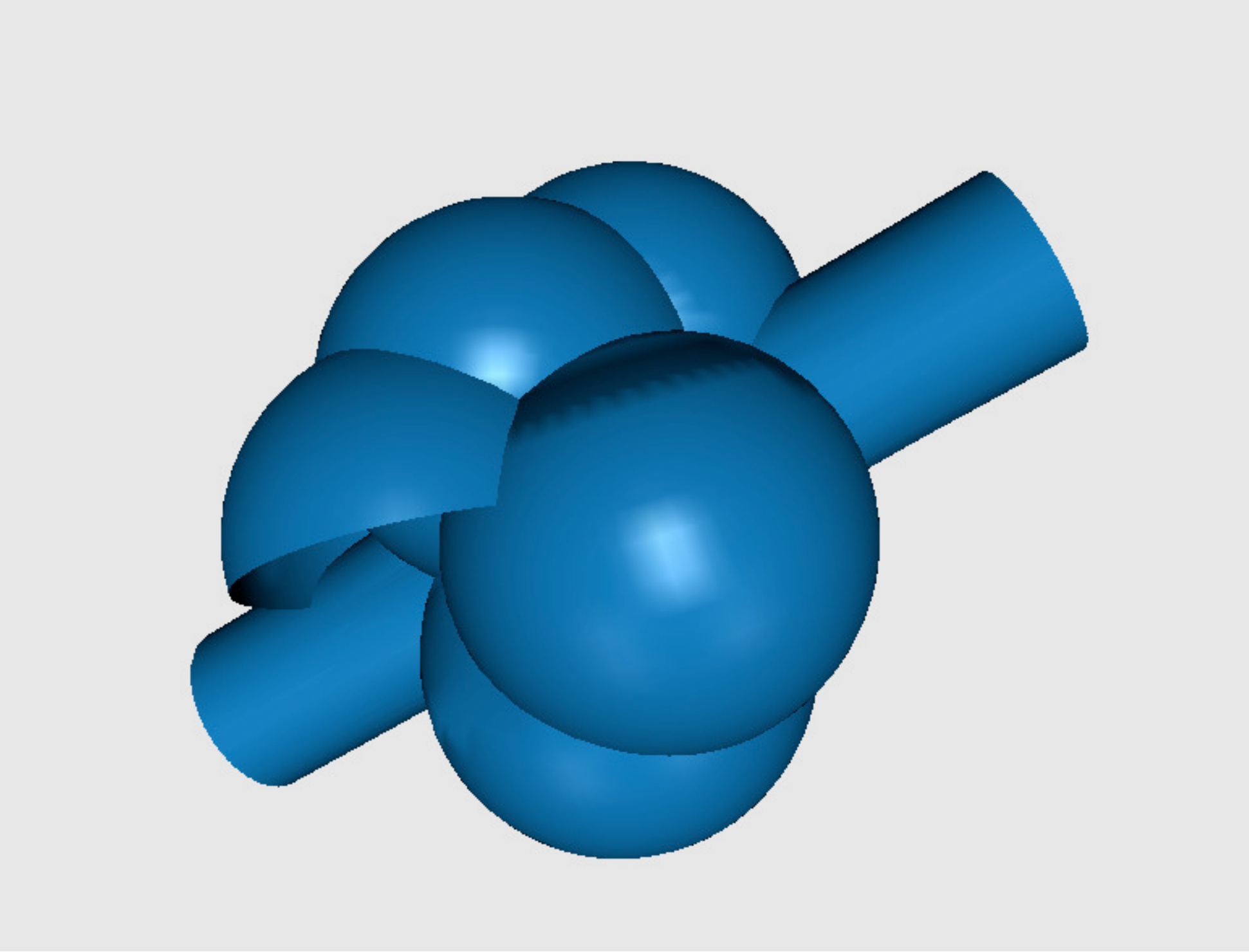}
\includegraphics[width=0.45\linewidth]{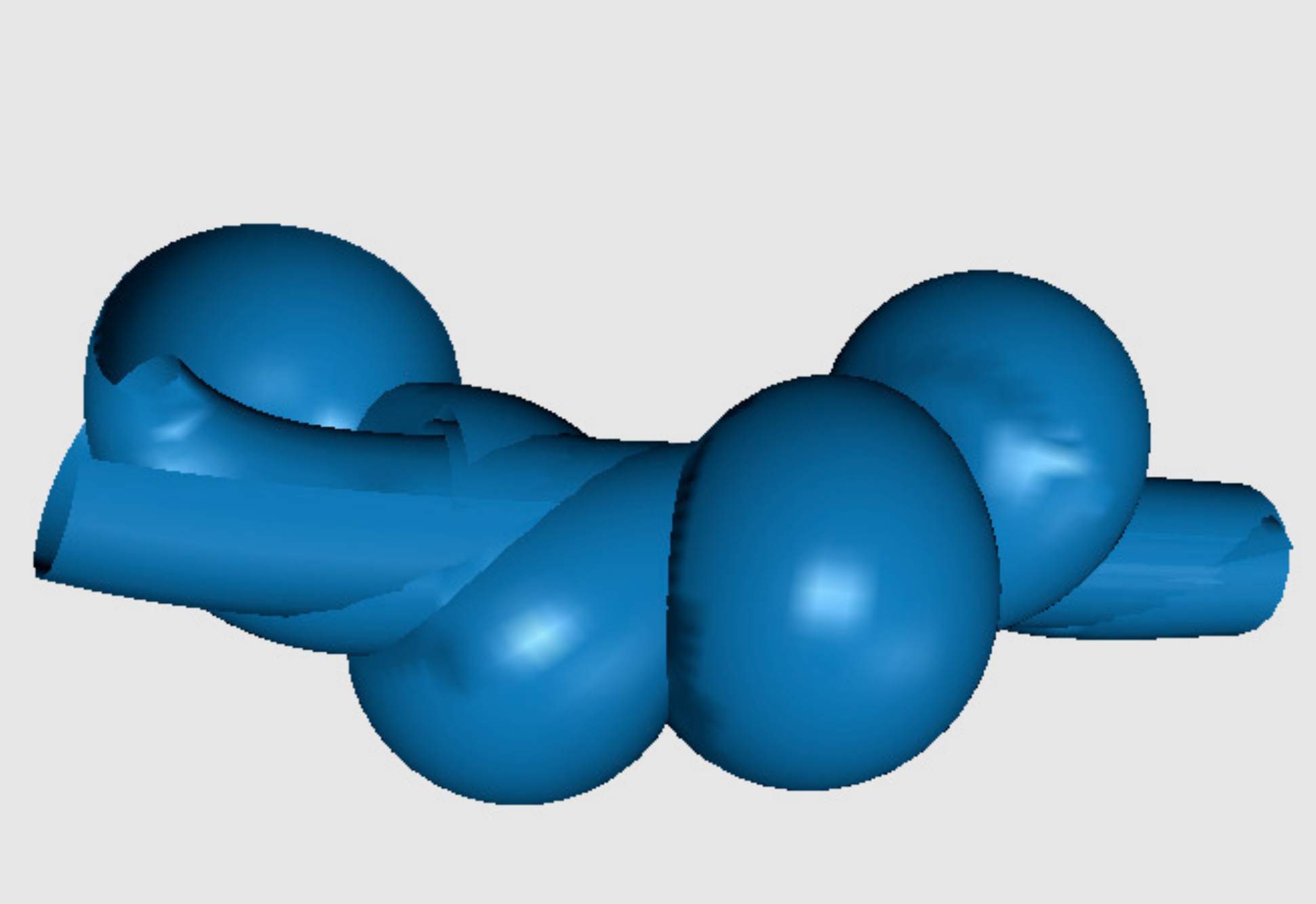}
\caption{$\mu$--Darboux transforms which are not  classical Darboux transforms.}
\end{center}
\end{figure}

We now compute the closed $\mu$--Darboux transforms of the standard cylinder $f$ which are those $\hat{f}$ satisfying
\[
\hat{f}(x,y)=\hat{f}(x,y+2\pi)\,.
\]
These are
given by parallel sections
$\varphi\in\Gamma(\widetilde{V/L})$ of $\nabla^\mu$ such that 
\begin{equation}
\label{eq:monodromy}
 \varphi(x,
y+2\pi) = \varphi(x,y) h\quad\text{with}\quad h\in\H_*\,.
\end{equation}
For $\mu$ on the unit circle $\nabla^\mu$ is a
quaternionic connection on a rank one bundle. Then all parallel
sections satisfy \eqref{eq:monodromy} and hence give rise to closed $\mu$-Darboux transforms. 
In particular, as in Theorem~\ref{thm:DT is CMC}, when $\mu\in S^1\setminus\{1\}$ we
obtain a translated cylinder $g = f+ N + \frac{b}{1-a}$. 

For $\mu$ not on the unit circle $\nabla^\mu$ is not 
quaternionic  but rather an ${\bf SL}(2,\C)$-connection on the complex rank two bundle $\widetilde{V/L}$.
Therefore any parallel section satisfying \eqref{eq:monodromy} must in fact have monodromy $h\in\C_*$ and is thus given by the eigenvectors of the
monodromy representation of 
$\nabla^\mu$. For $\mu\in \C_*$ these  parallel sections are
\[
\alpha^\pm = \begin{pmatrix} \emivh \\
p_\pm e^{\frac {iy}2}
\end{pmatrix}
e^{\pm w}
\]
and have monodromy
\begin{equation}
\label{eq:mono_cyl}
h_\pm^\mu = -e^{\mp\frac{\sqrt{2}b\pi}{2c}}\in\C_*\,,
\end{equation}
where we recall $p_{\pm}$ and $w$ from \eqref{eq:wp}.
Therefore, the prolongations of the sections $\varphi^\pm=e\alpha^\pm$
give closed $\mu$--Darboux transforms $\hat f^\pm$. Using
\eqref{eq:Tinvers} the $\mu$--Darboux transforms $\hat f^\pm$ are
rotations and translations of the original cylinder. Explicitly, 
\[
\hat f^\pm = f + T^\pm = f + T_0^\pm + j e^{iy} T_1^\pm
\]
 where
\[
T_0^\pm = \frac{2}{r^\pm}
\left(
\Re b - 2\frac{\bar p_\pm^2 i}{1+|p_\pm|^2}\Im a - \frac{1-|p_\pm|^2}{1+|p_\pm|^2}
i\Im b \right)\in\C\]
and
\[
T_1^\pm = \frac{2}{r^\pm}\left(\Re a -1 + \frac{1-|p_\pm|^2}{1+|p_\pm|^2}i\Im a - \frac{2p_\pm}{1+|p_\pm|^2} i \Im b)\right)\in\C
\]
with
\[
r^\pm = |a-1|^2+|b|^2 - 4 \,\Im((a-1)\bar b) \frac{\Im
  p_\pm}{1+|p_\pm|^2}\in\R\,.
\]
We call $\mu\in\C_*$ a  \emph{resonance point} of $f$ if the monodromies $h_\pm^\mu$ coincide. These are the points
\[
\mu_k=2k^2-1 -2k\sqrt{k^2-1}\,, \qquad k\in\Z\setminus\{0\}\,.
\]
Away from resonance points the monodromy of $\nabla^\mu$ is diagonalisable with two distinct
eigenvalues $h^\mu_+$ and $h^\mu_-$. Hence we obtain exactly two
$\mu$--Darboux transforms given by parallel sections of $\nabla^\mu$.

At the resonance points $\mu=\mu_k$ with $|k|>1$, by \eqref{eq:mono_cyl}, the monodromy
is 
\[
h^\mu:= h_+^\mu = h_-^\mu =-1\,.
\]
The space of parallel sections  
of $\nabla^\mu$ with monodromy $h=-1$
has dimension two and gives rise to a  $\CP^1$-- family of closed $\mu$--Darboux transforms.
A parallel section $\alpha = \alpha_{+}m_{+} + \alpha_{-} m_-$ for
$m_+=0$
or $m_-=0$ gives a rigid motion of the standard cylinder.
However for $m_+m_-\not=0$ the parallel section $\alpha $ yields a bubbleton
with $|k|$ lobes where the parameter  $\tfrac{m_+}{m_-}\in \CP^1$ rotates and slides the bubble
\cite{wente_sterling, darboux_isothermic}.
\begin{theorem} Every non--constant closed $\mu$--Darboux transform $\hat f:
  M\to\R^4$ of the standard cylinder $f$  is a rigid motion of $f$ provided
  $ \mu$ is not a resonance point $\mu_k$. At a resonance point $\mu_k$ we additionally obtain 
  a $\CP^1$--family of bubbletons with $|k|$ lobes. 
\end{theorem}
\begin{figure}[h]
\begin{center}
\includegraphics[width=0.45\linewidth]{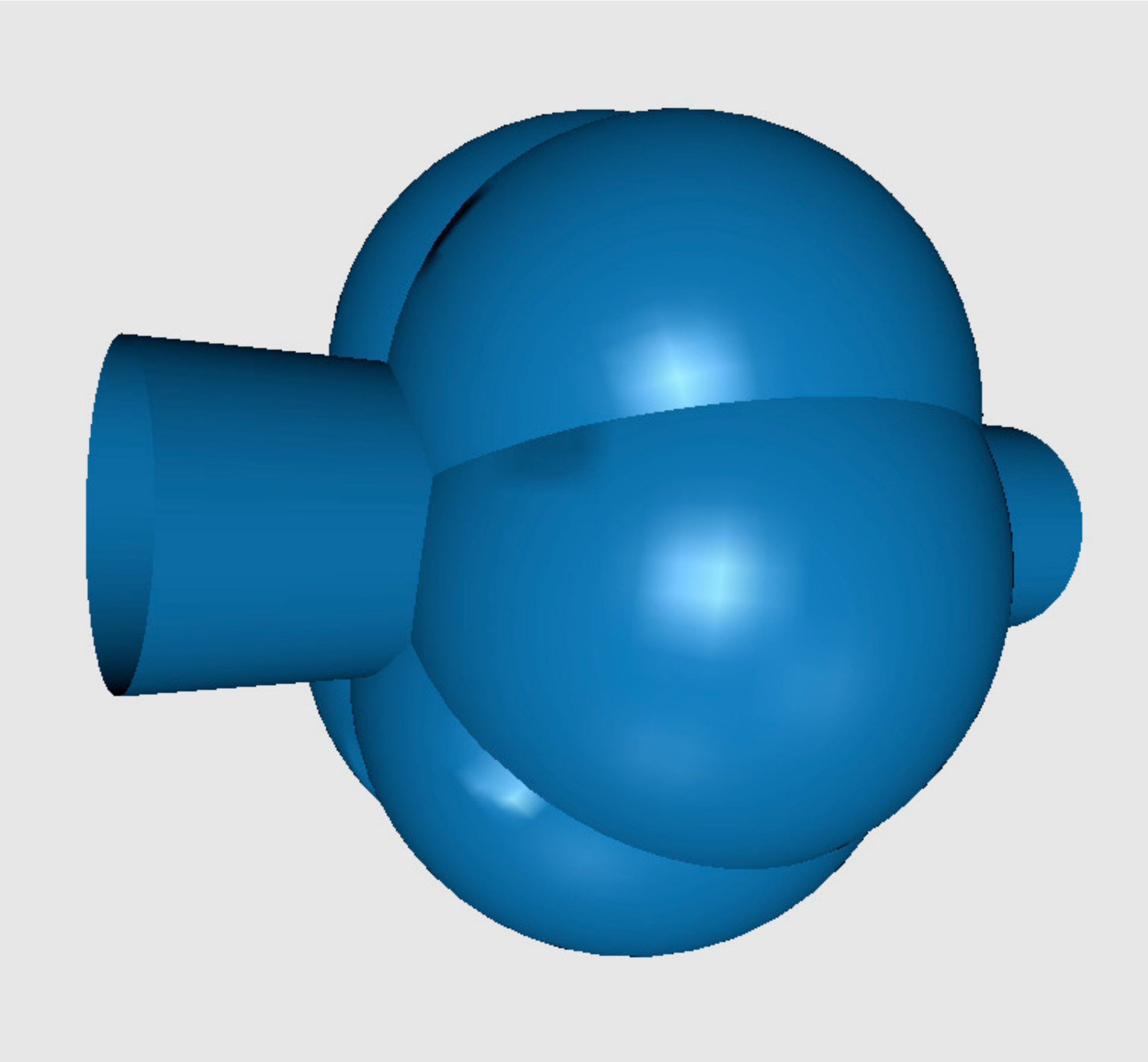}
\caption{Darboux transform at the resonance point $\mu_6$ with 6 lobes.}
\end{center}
\end{figure}
We now restrict our attention to those closed $\mu$--Darboux transforms of the standard cylinder which are also  classical Darboux transforms. By Theorem~\ref{thm:DT is not CDT} these are precisely the $\mu$--Darboux transforms for which $\mu\in\R_*\cup S^1$.
\begin{cor}
 The closed $\mu$--Darboux transforms of the standard cylinder $f:
  M\to\R^3$ which are also classical are:
\begin{enumerate}
\item  the dual surface $g= f+N$ of $f$ for $\mu\in S^1\setminus\{1\}$, up to a translation of $\R^3$ in the ambient $\R^4$;
\item  a translation of $f$ in $\R^3$ along the $i$--axis by $\pm\frac{\sqrt{2}i}{\sqrt{a-1}}$ for $\mu\in (0,\infty)$ not a resonance point;
\item  a rotation of $f$ in $\R^3$ along the $i$-axis  for $\mu\in (-\infty,0)$ not a resonance point;
\item a bubbleton with $|k|$ lobes for a resonance point $\mu=\mu_k$  with $|k|>1$;
\item  a constant point for $\mu=1$.
\end{enumerate}
\end{cor}
\begin{proof}
A straightforward computation, using $a^2 + b^2=1$, shows that 
\[
1+|p_\pm|^2 = \begin{cases} \frac{2b}{a-1}p_\pm & a >1\\
                            2 & a<1
                          \end{cases}\,,\quad
1-|p_\pm|^2 = \begin{cases} \pm \frac{2\sqrt{2}i}{c}p_\pm & a>1\\
         0 &a<1\,.
       \end{cases}
\]
     Thus the translational part of the Darboux transform
     $\hat f = f + T_0^\pm + je^{iy}T_1^\pm$, away from the resonance
     points, is given by
\[
T_0^\pm = \begin{cases} \mp\frac{\sqrt{2}i}{c} & a >1\\
                        0 & a<1
                      \end{cases}
\]                      
and the rotational part by
\[
T_1^\pm = \begin{cases}   0 & a >1\\
\frac{2}{a-1} \mp \frac{\sqrt{2}cib}{(a-1)^2} & a<1\,.
\end{cases}
\]

\end{proof}

Since for each $k\in\Z$ a parallel section of $\nabla^{\mu_k}$ gives a
holomorphic section with monodromy $h=-1$, we can also take linear
combinations of parallel sections $\alpha^\pm_{\mu_k}$ and $
\alpha^\pm_{\mu_l}$ of $\nabla^{\mu_k}$ and $\nabla^{\mu_l}$
respectively with $k\not=l$. This gives  a holomorphic section
\[
\alpha = m_+\alpha^+_{\mu_k} + m_-\alpha^-_{\mu_k} +
n_+\alpha^+_{\mu_l} + n_-\alpha^-_{\mu_l}\in
H^0(\widetilde{V/L})
\]
with monodromy $h=-1$ for each choice of $m_\pm, n_\pm\in\C$. The
prolongation of $\alpha$ gives a closed Darboux transform of $f$,
however $\alpha$ is in general not a parallel section for any
$\mu\in\C_*$. In particular
\[
\{\text{closed $\mu$--Darboux transforms}\}\subsetneq \{\text{closed Darboux transforms}\}\,.
\]
For a further discussion of the geometry of these examples see
\cite{rectangular}.
\begin{rem}
  The Darboux transformation satisfies Bianchi permutability
  \cite{conformal_tori}. In particular, a $\mu$--Darboux transform of
  a bubbleton is given algebraically \cite{habil} and we obtain
  multibubbletons at the resonance points. Away from the resonance
  points the $\mu$--Darboux transformation is again given by a rigid motion. 
\end{rem}
\begin{figure}[h]
\begin{center}
\includegraphics[width=0.45\linewidth]{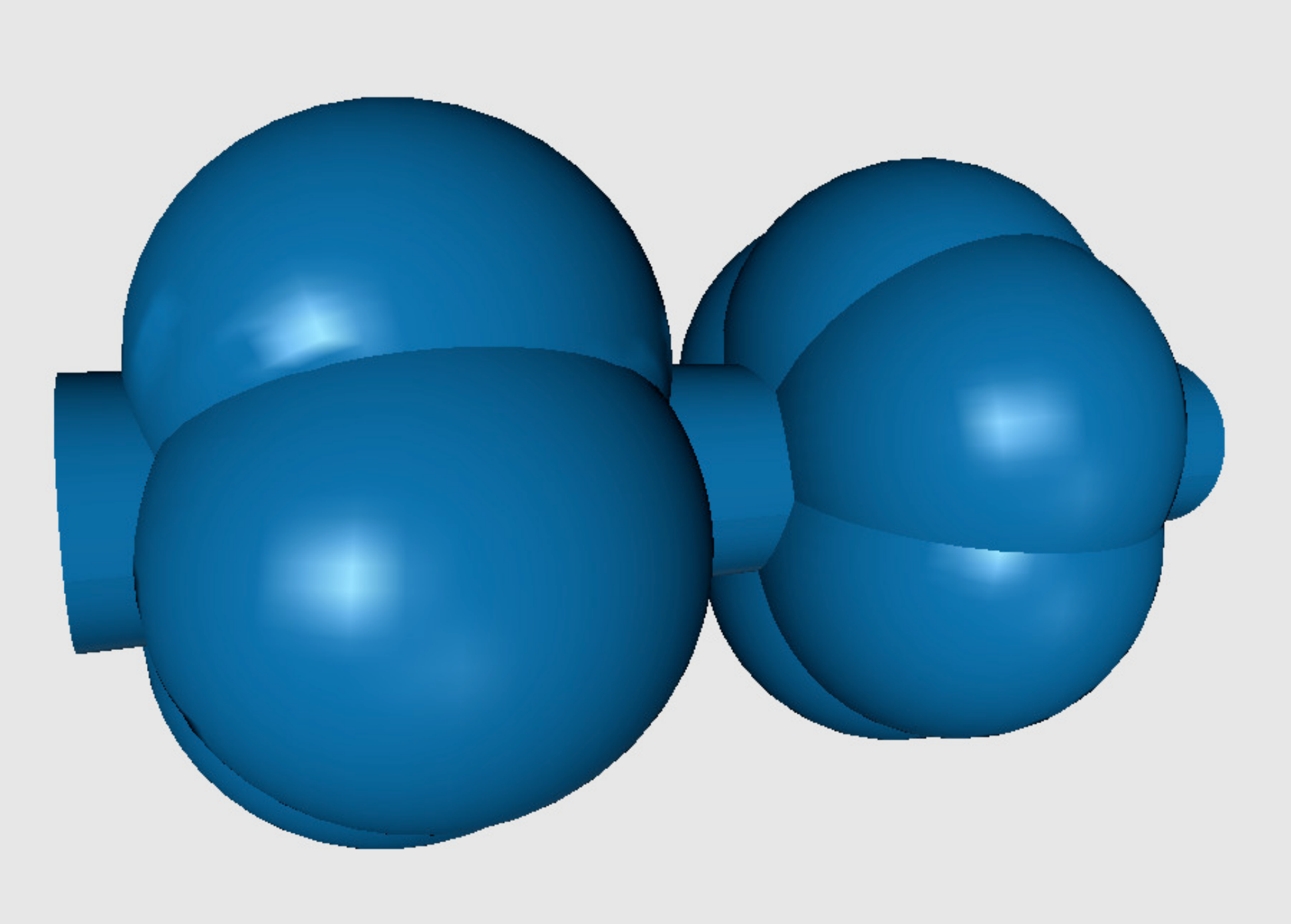}
\caption{Darboux transform of a bubbleton ($\mu=\mu_3$) at the resonance point $\mu_6$ ``adding'' 6 lobes.
}
\end{center}
\end{figure}

\end{appendix}

\bibliographystyle{amsplain}
\providecommand{\bysame}{\leavevmode\hbox to3em{\hrulefill}\thinspace}
\providecommand{\MR}{\relax\ifhmode\unskip\space\fi MR }
\providecommand{\MRhref}[2]{%
  \href{http://www.ams.org/mathscinet-getitem?mr=#1}{#2}
}
\providecommand{\href}[2]{#2}

\end {document}